\documentclass[12pt,reqno]{amsart}
\usepackage{amssymb,amsthm,amsmath,xypic, bbm,enumitem,verbatim}
\usepackage[all,2cell]{xy}
\usepackage[dvips]{color}
\usepackage{hyperref}
\UseTwocells
\numberwithin{equation}{section}

\newcommand{\MM}{{\mathbb M}}

\newcommand{\Z}{{\mathbb Z}}
\newcommand{\D}{{\mathbb D}}

\newcommand{\F}{{\mathcal F}}

\newcommand{\E}{{\mathfrak E}}
\newcommand{\EE}{{\mathcal E}}

\newcommand{\G}{{\mathcal G}}
\newcommand{\Gk}{{\mathfrak G}}
\newcommand{\M}{{\mathcal M}}
\newcommand{\Hc}{{\mathcal H}}
\newcommand{\Hg}{{\mathfrak H}}
\newcommand{\II}{{\mathfrak I}}

\newcommand{\V}{\mathcal V}
\newcommand{\U}{\mathcal U}
\newcommand{\X}{\mathcal X}
\newcommand{\Y}{\mathcal Y}
\newcommand{\Xb}{\mathbb X}
\newcommand{\Yb}{\mathbb Y}

\newcommand{\LC}{\mathcal L}

\newcommand{\CC}{\mathcal C}
\newcommand{\DD}{\mathcal D}
\newcommand{\SSS}{\mathcal S}

\newcommand{\mD}{\mathcal D}
\newcommand{\Sk}{\mathfrak S}
\newcommand{\cal}{\mathcal}
\newcommand{\Pro}{\mathcal P}
\newcommand{\I}{\mathcal I}
\newcommand{\T}{\mathfrak T}

\newcommand{\Qc}{\mathcal Q}
\newcommand{\Wk}{\mathfrak W}

\newcommand{\RR}{\mathfrak R}
\newcommand{\Ck}{\mathfrak C}
\newcommand{\Dk}{\mathfrak D}
\newcommand{\Kk}{\mathfrak K}
\newcommand{\Ak}{\mathfrak A}

\newtheorem{theorem}{Theorem}[section]
\newtheorem{corollary}[theorem]{Corollary}
\newtheorem{lemma}[theorem]{Lemma}
\newtheorem{proposition}[theorem]{Proposition}
\newtheorem{proposition-definition}[theorem]{Proposition - Definition}
\newtheorem{definition}[theorem]{Definition}
\newtheorem{remark}[theorem]{Remark}
\newtheorem{example}[theorem]{Example}

\allowdisplaybreaks

\begin{document}

\title[A sketch for derivators]{A sketch for derivators}
\author{Giovanni Marelli}

\begin{abstract}
We show first that derivators can be seen as models of a suitable homotopy limit 2-sketch. After discussing homotopy local $\lambda$-presentability of the 2-category of derivators, for some appropriate regular cardinal $\lambda$, as an application we prove that derivators of small presentation are homotopy $\lambda$-presentable objects. 

\end{abstract}

\address{University of Namibia, Department of Computing, Mathematical and Statistical Sciences, 
340 Mandume Ndemufayo Ave., 
13301 Windhoek (Namibia)}
\email{gmarelli@unam.na}
\maketitle

\section{Introduction}
Derivators were introduced by Grothendieck in his manuscript \cite{G} written between the end of 1990 and the beginning of 1991, 
though the term first appeared in his letter to Quillen \cite{GQ} of 1983.
Similar notions appeared, independently, in Heller's work \cite{H1} of 1988 with the name of homotopy theories, and later, in 1996, in Franke's paper \cite{F} with the name of systems of triangulated diagram categories. Then they were studied, for example, by
Heller himself \cite{H2}, Maltsiniotis \cite{M}, Cisinski \cite{C}, \cite{C2}, Cisinski and Neeman \cite{CN}, Keller \cite{K}, Tabuada  \cite{T}, Groth \cite{Gr}, Groth, Ponto and Shulman \cite{GPS}. 

A reason for proposing derivators is to provide a formalism improving that of triangulated categories. In fact, triangulated categories lack a good theory of homotopy limits and homotopy colimits, in the sense that, though they can be defined, they can not be expressed by means of an explicit universal property. An example of this is the non-functoriality of the cone construction. 
Since in the case of the derived category of an abelian category or the homotopy category of a stable model category or of a stable $(\infty,1)$-category, these construction can be made functorial, it means that when passing to the homotopy category the information for the construction of homotopy limits and homotopy colimits is lost. A derivator, as opposed to the homotopy (or derived) category, contains enough information to deal in a satisfactory way with homotopy limits and homotopy colimits. The idea in derivators is not only to consider the homotopy (or derived) category, but also to keep track of the homotopy (or derived) categories of diagrams and homotopy Kan extension between them. An advantage of working with derivators 
is also the possibility of describing them completely by means of the theory of 2-categories.

As proved by Cisinski \cite{C}, model categories give rise to derivators, yielding a pseudo-functor between the 2-category of model categories and the 2-category of derivators. Building on this and on Dugger's result \cite{D2} about presentation of combinatorial model categories, Renaudin \cite{R} proved that the pseudo-localization of the 2-category of combinatorial model categories at the class of Quillen equivalences is biequivalent to the 2-category of derivators of small presentation. These are defined by imposing, in a suitable sense, 
relations on a derivator associated to the model category of simplicial presheaves on a small category $\CC$, which  plays the role of a free derivator on $\CC$. 
In this sense, small presentation of derivators resembles finite presentation of modules over rings or of models of algebraic theories, when given in terms of generators and relations. However, in these last two cases, finite presentation can be characterized also intrinsically: finitely presented modules (or models) are those which represent functors preserving filtered colimits. The search for an analogous intrinsic formulation of small presentation for derivators has been the motivation for this paper.

The main result we have obtained is the construction of a homotopy limit 2-sketch whose homotopy models can be identified with derivators. A (homotopy) limit sketch is a way to describe a theory defined by means of (homotopy) limits. The 2-categories of (homotopy) models of (homotopy) limit 2-sketches are the (homotopy) locally presentable 2-categories. Therefore, the construction of a homotopy limit 2-sketch for derivators, besides providing some kind of algebraic description of derivators, supplies also a framework in which to discuss homotopy presentability. Indeed, as an application, we prove that derivators of small presentation are homotopy $\lambda$-presentable models, partially meeting our original motivation.

We summarize the content of the paper and present the results.  

In section \ref{derivators} we recall (right, left) derivators, as they were defined by Grothendieck \cite{G}, and we present Cisinski's result mentioned above. In this paper, in order to study presentability, we will assume that the 2-category of diagrams $\mathfrak{Dia}$ on which derivators are defined is small with respect to a fixed Grothendieck universe.

In section \ref{sketches}, we recall the definition of the weighted homotopy limit 2-sketch $\Sk$ and of its category of models. 
We explain, then, how to include pseudo-natural transformations as morphisms between models in a new 2-category of models $\mathfrak{hMod}_\Sk^{ps}$.

In section \ref{skder} we present our main result: we prove that the 2-category $\mathfrak{Der}^r$ of right derivators, cocontinuous pseudo-natural transformations (\ref{cc}) and modifications, is the 2-category of models of a weighted homotopy limit 2-sketch, whose construction is explicitly exhibited. 

\newtheorem*{Main}{Theorem \ref{Main}}
\begin{Main}
There exists a weighted homotopy limit 2-sketch $\Sk=(\Gk,\Pro)$ and a biequivalence from the 2-category $\mathfrak{hMod}_\Sk^{ps}$ to the 2-category $\mathfrak{Der}^r$.
\end{Main}

In section \ref{accessibility} we recall the theory of homotopy presentable categories, together with the notion of presentable object in the homotopic sense. We have: 

\newtheorem*{opls1}{Corollary \ref{opls1}}
\begin{opls1}
$\mathfrak{hMod}^{ps}_\Sk$ is a homotopy locally $\lambda$-presentable 2-category, where $\lambda$ is a regular cardinal bounding the size of every category in $\mathfrak{Dia}$.
\end{opls1}

In section \ref{finalres}, we prove first, in lemma \ref{repf}, after passing to a realized sketch, that representable models correspond to derivators defined by model categories of the form $sSet^{\CC^{op}}$, for some small category $\CC$. As an application, 
derivators can be reconstructed by means of homotopy $\lambda$-filtered colimits as follows:

\newtheorem*{ultimo7}{Corollary \ref{ultimo7}}
\begin{ultimo7}
Any right derivator is a homotopy $\lambda$-filtered colimit in $\mathfrak{Der}^r$ of $\lambda$-small homotopy 2-colimits of derivators of the form $\F(\CC)=\Phi(sSet^{\CC^{op}})$.
\end{ultimo7}

Finally, after recalling Renaudin's definitions and result on small presentability, we obtain:

\newtheorem*{main7}{Theorem \ref{main7}}
\begin{main7}
A derivator of small presentation is a homotopy $\lambda$-presentable object of $\mathfrak{Der}^r$.
\end{main7}

The author would like to thank Kuerak Chung for introducing the topic, Bernhard Keller for bringing this problem to his attention, Steve Lack and John Power for useful suggestions, Georges Maltsiniotis and Mike Shulman for useful comments.


\section{Derivators}
\label{derivators}
In this section we recall derivators as 
introduced by Grothendieck \cite[1]{G}. Derivators of small presentation, defined by Renaudin \cite[3.4]{R}, will be recalled instead in section \ref{finalres}. Besides these two references, introductions to derivators 
are found for instance in \cite{M}, \cite[1]{C} or \cite[1]{Gr}. 

We fix a Grothendieck universe $\U$ and we denote by $\mathfrak{Cat}$ the 2-category of $\U$-small categories, and by $\CC at$ the ordinary category underlying $\mathfrak{Cat}$. 

\begin{definition}
A category of diagrams, which we denote by $\mathfrak{Dia}$, is a full 2-subcategory of $\mathfrak{Cat}$ such that:
\begin{enumerate}
\item it contains the empty category, the terminal category $e$ and the category $\Delta_1=\mathbbm{2}$ associated to the ordered set $\{0<1\}$;
\item it is closed under finite coproducts and pullbacks;
\item it contains the overcategories $\CC/D$ and the undercategories $D\backslash\CC$ corresponding to any functor $u:\CC\rightarrow\DD$ and to any object $D\in\DD$;
\item it is stable under passage to the opposite category.
\end{enumerate}
\end{definition}

Examples of categories of diagrams are $\mathfrak{Cat}$ itself, the 2-category $\mathfrak{Cat}_f$ of finite categories, the 2-category of partially ordered sets 
or the 2-category of finite ordered sets.

In this paper we will assume that $\mathfrak{Dia}$ is $\U$-small, because, although the definitions regarding derivators do not depend on this, this hypothesis guarantees that all limits and colimits with which we will be concerned are $\U$-small. So we will assume the existence of a regular cardinal $\lambda$ such that all the categories in $\mathfrak{Dia}$ are $\lambda$-small.

\begin{definition}
A prederivator 
of domain $\mathfrak{Dia}$ is a strict 2-functor 
\begin{displaymath}
\D:\mathfrak{Dia}^{coop}\rightarrow\mathfrak{Cat}.
\end{displaymath}
\end{definition}

In other words, applying a prederivator $\D$ to the diagram
\begin{displaymath}
\xymatrix{
\CC \rrtwocell^u_v{~\alpha} & & \mD
}
\end{displaymath}
yields the diagram
\begin{displaymath}
\xymatrix{
\D(\mD) & & \D(\CC) \lltwocell^{v^\ast}_{u^\ast}{\alpha^\ast~}
}
\end{displaymath}
where we have set $u^\ast=\D(u)$, $v^\ast=\D(v)$ and $\alpha^\ast=\D(\alpha)$. 

\begin{example}
\rm
For any category $\CC\in\mathfrak{Dia}$, the representable 2-functor $\mathfrak{Dia}(-^{op},\CC)$ is a prederivator of domain $\mathfrak{Dia}$. 
Actually, any $\CC\in\mathfrak{Cat}$ defines a prederivator of domain $\mathfrak{Dia}$. 
\end{example}

We present now the definitions of derivator, right derivator and left derivator, as introduced by Grothendieck \cite{G}. There are other variants, which, however, we do not consider in this paper
(see, for instance, \cite[1]{CN}). 

\begin{definition}
\label{derivatordef}
A derivator is a prederivator $\D$ satisfying the following axioms.
\begin{enumerate}
\item For every $\CC_0$ and $\CC_1$ in $\mathfrak{Dia}$, the functor
\begin{equation*} 
\D(\CC_0\amalg\CC_1)\longrightarrow\D(\CC_0)\times\D(\CC_1), 
\end{equation*}
induced by the canonical inclusions $\CC_i\rightarrow\CC_0\amalg\CC_1$, is an equivalence of categories. 
Moreover, $\D(\varnothing)$ is equivalent to the terminal category $e$.
\item A morphism $f:A\rightarrow B$ of $\D(\CC)$ is an isomorphism if and only if, for any object $D$ of $\CC$, the morphism in $\D(e)$
\begin{equation*} 
c_D^\ast(f):c_D^\ast(A)\longrightarrow c_D^\ast(B)
\end{equation*}
is an isomorphism, where $c_D:e\rightarrow\CC$ denotes the constant functor at $D$.
\item For every $u:\CC\rightarrow\mD$ in $\mathfrak{Dia}$, the functor
\begin{displaymath}
u^\ast:\D(\mD)\longrightarrow\D(\CC)
\end{displaymath}
has both left and right adjoints
\begin{align}
 & u_!:\D(\CC)\longrightarrow\D(\mD) \label{hdi} \\
 & u_\ast:\D(\CC)\longrightarrow\D(\mD), \label{cdi}
\end{align}
called homological and cohomological direct image functor respectively.
\item Consider diagrams in $\mathfrak{Dia}$ of the form
\begin{displaymath}
\xymatrix{
D\backslash\CC \ar[r]^f  \ar[d]_t & \CC \ar[d]^u & & \CC/D \ar[r]^f  \ar[d]_t \drtwocell<\omit>{~\beta} & \CC \ar[d]^u \\
e \ar[r]_{c_D} & \DD \ultwocell<\omit>{\alpha~} & & e \ar[r]_{c_D} & \DD
}
\end{displaymath}
where $D\in\DD$, $t$ is the unique functor to the terminal category $e$, $f$ the obvious forgetful functor, $c_D$ the constant functor at $D$, $\alpha$ and $\beta$ the canonical natural transformations.
Apply $\D$ 
\begin{displaymath}
\xymatrix{
\D(D\backslash\CC) \drtwocell<\omit>{~\alpha^\ast} & \D(\CC) \ar[l]_{f^\ast} & & \D(\CC/D) & \D(\CC) \ar[l]_{f^\ast} \\
\D(e) \ar[u]^{t^\ast} & \D(\DD) \ar[l]^{c_D^\ast} \ar[u]_{u^\ast} & & \D(e) \ar[u]^{t^\ast} & \D(\DD) \ar[l]^{c_D^\ast} \ar[u]_{u^\ast} \ultwocell<\omit>{\beta^\ast~}
}
\end{displaymath}
and use axiom 3 to construct the Beck-Chevalley transformations 
\begin{align}
\alpha^\ast_{bc} & :t_! f^\ast\Rightarrow c_D^\ast u_! \label{hdi4} \\
\beta^\ast_{bc} & :c_D^\ast u_\ast\Rightarrow t_\ast f^\ast, \label{cdi4}
\end{align}
shown in the diagrams
\begin{displaymath}
\xymatrix{
\D(D\backslash\CC) \ar[d]_{t_!} & \D(\CC) \ar[d]^{u_!} \ar[l]_{f^\ast} & & \D(\CC/D) \ar[d]_{t_\ast} & \D(\CC) \ar[d]^{u_\ast} \ar[l]_{f^\ast} \dltwocell<\omit>{\beta^\ast_{bc}~} \\
\D(e) \urtwocell<\omit>{~~\alpha^\ast_{bc}} & \D(\DD) \ar[l]^{c_D^\ast} & & \D(e) & \D(\DD) \ar[l]^{c_D^\ast}
}
\end{displaymath}
and given respectively by the composites
\begin{align*}
t_! f^\ast & \Rightarrow t_! f^\ast u^\ast u_! \Rightarrow t_! t^\ast c_D^\ast u_! \Rightarrow c_D^\ast u_! \\ 
c_D^\ast u_\ast & \Rightarrow t_\ast t^\ast c_D^\ast u_\ast \Rightarrow t_\ast f^\ast u^\ast u_\ast \Rightarrow  t_\ast f^\ast.
\end{align*}
Then the natural transformations $\alpha^\ast_{bc}$ and $\beta^\ast_{bc}$ are isomorphisms.
\end{enumerate}
\end{definition}

\begin{definition}
\label{weakder}
A right derivator is a prederivator such that:
\begin{itemize}
\item it satisfies axioms 1 and 2; 
\item it admits homological direct image functors $u_!$ for any functor $u$ in $\mathfrak{Dia}$; 
\item every $\alpha^\ast_{bc}$ as in \eqref{hdi4} is an isomorphism.
\end{itemize}
A left derivator is defined in an analogous way. 
\end{definition}


\begin{example}
\rm
\label{modcatder}
Let $\M$ be a model category
and $W$ the class of its weak equivalences. The prederivator ${\rm Ho}[-^{op},\M]$, which on objects $\CC\in\mathfrak{Dia}$ is defined as the homotopy category
\begin{equation*}
{\rm Ho}[\CC^{op},\M]=[\CC^{op},\M][W_{\CC}^{-1}],
\end{equation*} 
where $W_{\CC}$ is the class of objectwise weak equivalences,
defines a derivator. Its value on the terminal category $e$ is just the homotopy category ${\rm Ho}(\M)$ of $\M$. Its complete definition and the proof that it does define a derivator is the subject of \cite{C}. 
\end{example}

We use pseudo-natural transformations to define 1-morphisms of 
derivators.

\begin{definition}
\label{morphisms}
A morphism of prederivators $\theta:\D_1\rightarrow\D_2$ 
is a pseudo-natural transformation $\theta:\D_1\Rightarrow\D_2$.
%
\end{definition}

Explicitly, a pseudo-natural transformation $\theta:\D_1\Rightarrow\D_2$ consists of the following data:
\begin{enumerate}
\item for any $\CC\in\mathfrak{Dia}$, a functor 
\begin{displaymath}
\theta_\CC:\D_1(\CC)\longrightarrow\D_2(\CC);
\end{displaymath}
\item for any $\CC$, $\DD$ and $u:\CC\rightarrow\DD$ in $\mathfrak{Dia}$, 
an isomorphism
\begin{displaymath}
\beta^\theta_{\CC\DD_u}\equiv\beta^\theta_u:u^\ast_2\circ\theta_\DD\Rightarrow\theta_\CC\circ u^\ast_1,
\end{displaymath}
where $u^\ast_i=\D_i(u)$ for $i=1,2$, which is
natural in $u$, 
that is, for any $\alpha:u\Rightarrow v$ in $\mathfrak{Dia}$ the diagram
\begin{displaymath}
\xymatrix{
v_2^\ast\circ\theta_\DD \ar[r]^{\beta^\theta_v}  \ar[d]_{\alpha_2^\ast\ast\theta_\DD} & \theta_\CC\circ v_1^\ast \ar[d]^{\theta_\CC\ast\alpha^\ast_1} \\
u_2^\ast\circ\theta_\DD \ar[r]_{\beta^\theta_u} & \theta_\CC\circ u_1^\ast 
}
\end{displaymath}
is commutative;
\end{enumerate}
these data are required to fulfill the following coherence conditions 
\begin{align*}
\beta^\theta_{1_\CC} & =1_{\theta_\CC} \\
\beta^\theta_{vu} & =(\beta^\theta_{u}\ast v_1^\ast)\circ(u_2^\ast\ast\beta^\theta_{v})
\end{align*}
for any composable $u$ and $v$.

\begin{definition}
\label{cc}
A morphism of right derivators $\theta:\D_1\rightarrow\D_2$ is cocontinuous if it is
compatible with the homological direct image functors, namely,
for every $u$ in $\mathfrak{Dia}$ the Beck-Chevalley transform 
\begin{align*}
\beta^\theta_{u_!} & :u_{2!}\circ\theta_\CC\Rightarrow\theta_\DD\circ u_{1!}
\end{align*}
is an isomorphism.

Continuous morphisms of (left) derivators are defined in an analogous way. 
\end{definition}

It remains to define 2-morphisms of derivators.

\begin{definition}
\label{mod}
Given two (pre)derivators $\D_1$ and $\D_2$ 
and two morphisms $\theta_1$, $\theta_2:\D_1\rightarrow\D_2$, a 2-morphism $\lambda:\theta_1\rightarrow\theta_2$ is a modification $\lambda:\theta_1\Rrightarrow\theta_2$ between the underlying pseudo-natural transformations.
\end{definition}

Explicitly, a modification $\lambda:\theta_1\Rrightarrow\theta_2$ consists of a family of natural transformations 
\begin{equation*}
\lambda_\CC:\theta_{1\CC}\Rightarrow\theta_{2\CC}
\end{equation*}
for any $\CC\in\mathfrak{Dia}$, such that for every $u:\CC\rightarrow\DD$ of $\mathfrak{Dia}$ the diagram
\begin{gather}
\begin{aligned}
\xymatrix{
u_2^\ast\circ\theta_{1\CC} \ar[r]^{\beta^{\theta_1}_u}  \ar[d]_{u_2^\ast\ast\lambda_\CC} & \theta_{1\DD}\circ u_1^\ast \ar[d]^{\lambda_\DD\ast u_1^\ast} \\
u_2^\ast\circ\theta_{2\CC} \ar[r]_{\beta^{\theta_2}_u} & \theta_{2\DD}\circ u_1^\ast 
}
\end{aligned}
\label{modi}
\end{gather}
is commutative.

We organize what has been introduced so far into the following 2-categories:
\begin{enumerate}
\item $\mathfrak{PDer}$ the 2-category of prederivators, morphisms of prederivators and 2-morphisms,
\item $\mathfrak{Der}^r$ the 2-category of right derivators, cocontinuous morphisms and 2-morphisms, 
\item $\mathfrak{Der}^l$ the 2-category of left derivators, continuous morphisms and 2-morphisms, 
\item $\mathfrak{Der}^{rl}$ the 2-category of derivators, continuous and cocontinuous morphisms and 2-morphisms, 
\item $\mathfrak{Der}_{ad}$ the 2-category of derivators, 
morphisms of derivators whose components have right adjoints, and modifications. 
\end{enumerate}

We conclude this section by telling more about the relationship between derivators and model categories outlined in example \ref{modcatder}.
Let $\mathfrak{ModQ}$ denote the 2-category of model categories, 
left Quillen functors and 
natural transformations. 
Cisinski proved in \cite{C} that the map in example \ref{modcatder} 
\begin{align*}
ob\mathfrak{ModQ} & \longrightarrow ob\mathfrak{Der}_{ad} \\
\M & \longmapsto {\rm Ho}[-^{op},\M]
\end{align*}
extends to 1-morphisms and 2-morphisms: he showed that 
a Quillen adjunction $F:\M_1\rightleftarrows\M_2:G$ induces for any $\CC\in\mathfrak{Dia}$ an adjunction of total derived functors
\begin{displaymath}
{\bf L}\tilde{F}:{\rm Ho}[\CC^{op},\M_1]\rightleftarrows {\rm Ho}[\CC^{op},\M_2]:{\bf R} \tilde{G},
\end{displaymath}
where $\tilde{F}$ and $\tilde{G}$ act by composing with $F$ and $G$ respectively,
and so it defines a pair of adjoint morphisms between the corresponding derivators. 

\begin{theorem}
\label{cisinski}
The construction above defines a pseudo-functor 
\begin{equation*}
\Phi:\mathfrak{ModQ}\rightarrow\mathfrak{Der}_{ad}
\end{equation*}
taking Quillen equivalences to equivalences of derivators.
\end{theorem}

We will use the symbol $\Phi(\M)$ for the derivator ${\rm Ho}[-^{op},\M]$ constructed from a model category $\M$.

We will recall other facts about derivators, especially the definition of small presentation, in section \ref{finalres}.

\section{Sketches}
\label{sketches}
Sketches, introduced by Ehresmann \cite{E}, are a way of presenting a theory which can be defined by means of limits and colimits. 
It turns out that the categories of models of sketches can be characterized intrinsically as the accessible categories (Lair \cite[1-2]{Lr}), and, in particular, the categories of models of limit sketches are the locally presentable categories (Gabriel and Ulmer \cite{GU}). 

Though the underlying idea is the same, there are different types of sketches, depending on the type of limits and colimits which define the theory we want to describe. 
In this section we recall, in some detail, homotopy limit 2-sketches: in fact, in section \ref{skder} we will prove that derivators can be identified, up to equivalence, with the 
homotopy models of a sketch of this type. The 2-category of homotopy models, pseudo-natural transformations and modifications is then homotopy locally presentable.
As an application, in section \ref{finalres}, we will use this framework to study small presentability of derivators.

Homotopy limit sketches were proposed by Rosick\'y \cite{Ro} with the purpose of extending rigidification results of Badzioch \cite{Ba} and Bergner \cite{Be} to finite limit theories. Lack and Rosick\'y in \cite{LR} proved that the $\V$-categories of homotopy models of homotopy limit $\V$-sketches can be characterized as the homotopy locally presentable $\V$-categories. 

We will consider only the case $\V=\CC at$, since this is the one of derivators.
Recall that $\CC at$ has a model structure, known as the standard model structure, where weak equivalences are the equivalences of categories, and fibrations are the isofibrations; 
this model structure is combinatorial, 
all objects are fibrant and, assuming the axiom of choice, also cofibrant, moreover, $\mathfrak{Cat}$ becomes a monoidal model 2-category (in the sense of \cite[A.3.1.2]{Lu}). 


If $\E$ is a small 2-category, then the category underlying $[\E,\mathfrak{Cat}]$, endowed with the projective model structure, is also a combinatorial model category, whose 
cofibrant objects can be characterized as follows.
Recall that the inclusion
\begin{equation*}
i:[\E,\mathfrak{Cat}]\hookrightarrow \Pro s(\E,\mathfrak{Cat})
\end{equation*}
has a left adjoint $\Qc$ (see \cite[2.2]{Bk}), where $\Pro s(\E,\mathfrak{Cat})$ denotes the 2-category of 2-functors $\E\rightarrow\mathfrak{Cat}$, pseudo-natural transformations and modifications. 
Thus, for 2-functors $G,H:\E\rightarrow\mathfrak{Cat}$, there is a natural isomorphism of categories
\begin{equation}
\label{coheq}
\mathfrak[\E,\mathfrak{Cat}](\Qc G,H)\cong\Pro s(\E,\mathfrak{Cat})(G,H).
\end{equation}
The counit and unit computed at a functor $G:\E\rightarrow\mathfrak{Cat}$ are given by a 2-natural transformation
$\varepsilon_G:\Qc(G)\rightarrow G$ and a pseudo-natural transformation $\eta_G:G\rightarrow\Qc(G)$ respectively.
One of the triangle equations tells us that $\varepsilon_G\circ\eta_G=1_G$. Since $\eta_G\circ\varepsilon_G\cong1_G$ (see \cite[4.2]{BKP}), it follows that $\Qc G$ and $G$ are equivalent in $\Pro s(\E,\mathfrak{Cat})$. If $\varepsilon$ has a section in $[\E,\mathfrak{Cat}]$, then $\Qc G$ and $G$ are equivalent also in $\mathfrak[\E,\mathfrak{Cat}]$ and $G$ is said to be flexible (see \cite[4.3]{La} and \cite[4.7]{BKP}). As proved in \cite[4.12]{La1}, flexible 2-functors are exactly the cofibrant objects of $[\E,\mathfrak{Cat}]$ with respect to the projective model structure, and $\Qc G$ is indeed a cofibrant replacement of $G$. 


\begin{definition}
\label{hwl}
Let $\Gk$ be a 2-category, $F:\E\rightarrow\Gk$ and $G:\E\rightarrow\mathfrak{Cat}$ be 2-functors, where $\E$ is a small 2-category. Assume $G$ is a cofibrant object of the category $[\E,\mathfrak{Cat}]$ endowed with the projective model structure. The homotopy 2-limit of $F$ weighted by $G$ exists when there is an object $\{G,F\}_h\in\Gk$ and for every object $\DD$ of $\Gk$ an equivalence of categories
\begin{equation}
\label{defhwl}
\Gk(\DD,\{G,F\}_h)\longrightarrow[\E,\mathfrak{Cat}](G,\Gk(\DD,F-))
\end{equation}
which is 2-natural in $\DD$.
\end{definition}

In a similar way we define the homotopy 2-colimit 
$G\star_h F$ of $F:\E\rightarrow\Gk$ weighted by $G:\E^{op}\rightarrow\mathfrak{Cat}$ by replacing formula (\ref{defhwl}) with 
\begin{equation*}
\Gk(G\star_h F,\DD)\longrightarrow[\E,\mathfrak{Cat}](G,\Gk(F-,\DD)).
\end{equation*}

The following definitions are from from \cite[2]{Ro}. 

\begin{definition}
\label{enrsk}
A weighted limit 2-sketch is a pair $\Sk=(\Gk,\Pro)$ where:
\begin{enumerate}
\item $\Gk$ is a small 2-category;
\item $\Pro$ is a set of 2-cones, that is, quintuples $(\E,F,G,\LC,\gamma)$ where $\E$ is a small 2-category, the diagram $F:\E\rightarrow\Gk$ and the weight $G:\E\rightarrow\mathfrak{Cat}$ are 2-functors, the vertex $\LC$ is an object of $\Gk$ and $\gamma:G\Rightarrow\Gk(\LC,F-)$ is a 2-natural transformation.
\end{enumerate}
A weighted homotopy limit 2-sketch is a weighted limit 2-sketch $\Sk=(\Gk,\Pro)$ with all weights cofibrant.
\end{definition}

\begin{definition}
\label{hm}
A homotopy model of a weighted homotopy limit 2-sketch $\Sk$ is a 2-functor $\MM:\Gk\rightarrow\mathfrak{Cat}$ transforming the cones of $\Pro$ into weighted homotopy 2-limits. We denote by $\mathfrak{hMod}_{\Sk}$ the full 2-subcategory of $[\Gk,\mathfrak{Cat}]$ spanned by the homotopy models of the weighted homotopy limit 2-sketch $\Sk$. 
\end{definition}

The 2-categories of the form $\mathfrak{hMod}_{\Sk}$ for some weighted homotopy limit 2-sketch $\Sk$ are the homotopy locally presentable 2-categories: this fact \cite[9.14(1)]{LR} is a consequence of \cite[9.10]{LR} (and, actually, holds for a more general $\V$). We will return to these results and to homotopy locally presentable 2-categories in \ref{summaryLR}.

To recover morphisms of derivators, we have to consider pseudo-natural transformations as morphisms between homotopy models. This motivates the following definition.

\begin{definition}
\label{mskkk1nh}
If $\Sk$ is a weighted homotopy limit 2-sketch, we define $\mathfrak{hMod}_\Sk^{ps}$ to be the full 2-subcategory of $\Pro s(\Gk,\mathfrak{Cat})$ spanned by the homotopy models.
\end{definition}

\section{A sketch for derivators}
\label{skder}

In this section we prove, by giving an explicit construction, that $\mathfrak{Der}^r$ is the 2-category $\mathfrak{hMod}_\Sk^{ps}$ of homotopy models of a homotopy limit 2-sketch $\Sk$. Analogous results hold for $\mathfrak{Der}^{l}$ and $\mathfrak{Der}^{rl}$, however, here we consider just the case of $\mathfrak{Der}^r$, since this is the one relevant to study of presentability of derivators. 

We recall that a biequivalence between 2-categories is a pseudo-functor which is 2-essentially surjective (surjective on objects up to equivalence), and a local equivalence (essentially full on 1-morphisms and full and faithful on 2-morphisms), see \cite[1.1.4-5]{R} and \cite[1.5.13]{Le}.

\begin{theorem}
\label{Main}
There exists a weighted homotopy limit 2-sketch $\Sk=(\Gk,\Pro)$ and a biequivalence from the 2-category $\mathfrak{hMod}_\Sk^{ps}$ to the 2-category $\mathfrak{Der}^r$.
\end{theorem}

Since the proof is long, we split it into several parts.

\subsection{Idea of the proof}
\label{ideap}

The proof consists of two parts: the first, from subsection \ref{prederivators} to \ref{resumesk}, contains the construction of a homotopy limit 2-sketch $\Sk=(\Gk,\Pro)$, and the second, in subsection \ref{verification}, the verification that the 2-category $\mathfrak{hMod}_\Sk^{ps}$ is indeed $\mathfrak{Der}^r$.

The construction of $\Sk$ will be carried out as follows.
After providing a 2-sketch for prederivators $(\Gk,\Pro)$ in subsection \ref{prederivators}, we will proceed by steps 
capturing, in subsections \ref{first}, \ref{secondd}, \ref{third} and \ref{fourth}, each of the four axioms for derivators. 
More precisely, we will adjoin to $\Gk$, at each step, 
new elements and commutative diagrams, and we will enlarge $\Pro$ with new cones, in order to express by means of these the axioms for derivators; then, 
we will redefine $\Gk$ as the free 2-category on these data and on the commutativity conditions already in $\Gk$ (see remark \ref{yyu} below). 
Observe that cones in $\Pro$ are used to capture only axiom 1 and 2. 

\begin{remark}
\label{yyu}
\rm 
The free construction we use to adjoin new elements to $\Gk$ generalizes the analogous construction for ordinary categories (see \cite[5.1]{B1}), replacing ordinary graphs with 2-graphs. A 2-graph is a graph ``enriched'' over the category of small graphs, that is, it is given by a set of vertices and a family of ordinary graphs, one for every pair of vertices (see \cite[1.3.1]{Le} for the precise definition).
If $2\G r$ denotes the category of 2-graphs and morphisms of 2-graphs, and $2\CC at$ the category of 2-categories whose underlying 2-graph belongs to $2\G r$ and 2-functors, then the forgetful functor $2\CC at\rightarrow 2\G r$ is monadic (see \cite[D]{Le}).

When a 2-graph contains elements already composable or relations among them, we would like that the free 2-category constructed over it preserves such data.
As usual, the idea is to consider, in the given 2-graph, pairs formed by finite sequences of horizontally or vertically composable 2-cells in a prescribed order, sharing horizontal sources and targets, and to require that the components of each pair become equal in the free 2-category. Such pairs, called commutativity conditions, are defined rigorously by Power and Wells \cite[2.5]{PW}, in terms of labeled pasting schemes, called pasting diagrams in \cite{St1}. The proof that pasting 2-cells is well-defined in any 2-category is the subject of \cite{PP}, of which a brief survey is found in \cite[2]{PB}.
Denoting by $c2\G r$ the category whose objects are 2-graphs with a set of commutativity conditions and whose morphisms are morphisms of 2-graphs preserving commutativity conditions, a free construction, left adjoint to the forgetful functor $2\CC at\rightarrow c2\G r$, is provided in Street's paper \cite[5]{St1} in terms of ``presentations" of 2-categories.

When a 2-graph $\Gk$ is built from a 2-category $\Ck$ by adjoining new symbols, as in our case, we refer to all the relations among elements of $\Ck$ determined by the 2-category structure on $\Ck$ as the commutativity conditions defined by $\Ck$ . 
\end{remark}

The first step consists in providing a sketch for prederivators. 

\subsection{Prederivators}
\label{prederivators}

Let $\Gk=\mathfrak{Dia}^{op}$ and set $\Pro=\varnothing$.
A homotopy model with values in $\mathfrak{Cat}$ is a 2-functor $\D:\Gk\rightarrow\mathfrak{Cat}$ with domain $\mathfrak{Dia}^{op}$, in other words, a prederivator of domain $\mathfrak{Dia}$. Therefore $\Sk=(\Gk,\Pro)$ 
is a homotopy limit 2-sketch whose 2-category $\mathfrak{hMod}_\Sk^{ps}$ of homotopy models 
in $\mathfrak{Cat}$ 
is the 2-category $\mathfrak{PDer}$ of prederivators.\\

The next steps are concerned with including into the sketch the axioms for derivators.

\subsection{Axiom 1}
\label{first}

Let $\Gk=\mathfrak{Dia}^{op}$ and define $\Pro$ to be the family of cones in $\mathfrak{Dia}^{op}$ of the form
\begin{gather}
\begin{aligned}
\xymatrix{
 & \CC_0\amalg\CC_1 \ar[dl]_{s_{\CC_0}} \ar[dr]^{s_{\CC_1}} & \\
\CC_0 & & \CC_1,  
}
\end{aligned}
\label{diaprod2}
\end{gather}
corresponding to cocones for the coproducts $\CC_0\amalg\CC_1$ in $\mathfrak{Dia}$, for any pair of objects $\CC_0$ and $\CC_1$. Therefore, $s_{\CC_0}$ and $s_{\CC_1}$ are the arrows in $\mathfrak{Dia}^{op}$ corresponding to the canonical morphisms of the coproduct $\CC_0\amalg\CC_1$ taken in $\mathfrak{Dia}$. With the notation of definition \ref{enrsk} we can write these cones as
\begin{equation}
\label{prodcon}
(\{0,1\},F,\delta_e,\CC_0\amalg\CC_1,(s_{\CC_0},s_{\CC_1})),
\end{equation}
where $\{0,1\}$ is the discrete 2-category with two objects, 
$F:\{0,1\}\rightarrow\Gk$ is the 2-functor mapping $i$ to $\CC_i$, for $i=0,1$, $\delta_e:\{0,1\}\rightarrow\mathfrak{Cat}$ is the constant 2-functor at the terminal category $e$ (which is clearly cofibrant), $\CC_0\amalg\CC_1$ denotes the product of $\CC_0$ and $\CC_1$ in $\mathfrak{Dia}^{op}$ (the coproduct in $\mathfrak{Dia}$) and $s_{\CC_i}:\CC_0\amalg\CC_1\rightarrow\CC_i$ are the canonical projections. 

Since models take the product cones \eqref{prodcon} to product cones in $\mathfrak{Cat}$, they fulfill the first part of axiom 1. To capture completely axiom 1, we have to include into $\Pro$ the cone $\varnothing$ with vertex the empty category over the empty diagram, thus forcing $\D(\varnothing)\simeq e$
 
Observe that $\Pro$ is a set, as we have assumed that $\mathfrak{Dia}$ is small for the fixed universe $\U$.

\subsection{Axiom 2}
\label{secondd}
To capture axiom 2 we need first a reformulation of it in terms of limits. As an intermediate step, we recast it as follows. 

\begin{lemma}
\label{second}
A prederivator $\D$ satisfies axiom 2 if and only if, 
for any $\CC\in\mathfrak{Dia}$, the family of functors $\D(c_D):\D(\CC)\rightarrow\D(e)$ induced by the constant functors $c_D:e\rightarrow\CC$ at $D\in\CC$, is jointly conservative, that is, the induced functor 
\begin{equation*}
\D(\CC)\rightarrow\Pi_{D\in\CC}\D(e)
\end{equation*} 
is conservative.
\end{lemma}

Conservative functors can be described as follows. Consider a functor $f:A\rightarrow B$. Denote by $A^\mathbbm{2}$ and $B^\mathbbm{2}$ the categories of arrows of $A$ and $B$ respectively, seen as categories of functors, where $\mathbbm{2}=\Delta^1$ is the category corresponding to the ordered set $\{0<1\}$. Let $c_A:A\rightarrow A^\mathbbm{2}$ and $c_B:B\rightarrow B^\mathbbm{2}$ denote the canonical inclusions. Let $f^{\mathbbm 2}:A^\mathbbm{2}\rightarrow B^\mathbbm{2}$ be the functor induced by $f$ via composition. With these data, consider the diagram
\begin{gather}
\begin{aligned}
\xymatrix{
A \ar[r] \ar[d]_{f} \ar[r]^{c_A} & A^{\mathbbm 2} \ar[d]^{f^{\mathbbm 2}} \\
B \ar[r]_{c_B} & B^{\mathbbm 2}
}
\end{aligned}
\label{consdia}
\end{gather}
in the 2-category $\mathfrak{Cat}$.

\begin{lemma}
\label{sad}
A functor $f:A\rightarrow B$ is conservative if and only if the commutative diagram \ref{consdia} 
is a bilimit 
in $\mathfrak{Cat}$.
\end{lemma}
We recall the notion of bilimit: if $F:\E\rightarrow\Gk$ and $G:\E\rightarrow\mathfrak{Cat}$ are 2-functors, where $\E$ is a 
small 2-category, the bilimit of $F$ weighted by $G$ exists when there is an object $\{G,F\}_b\in\Gk$ and for every object $\DD$ in $\Gk$ an equivalence in $\cal{C}at$
\begin{equation*}
\label{lhs}
\Gk(\DD,\{G,F\}_b)\simeq\Pro s(\E,\mathfrak{Cat})(G,\Gk(\DD,F-))
\end{equation*}
natural in $\DD$. 

Notice, however, that by the isomorphism \eqref{coheq}, any bilimit $\{G,F\}_b$ is equivalent to the weighted homotopy limit $\{\Qc G,F\}_h$, where $\Qc G$ is a cofibrant replacement of $G$, so that a bilimit is a special case of weighted homotopy limit (definition \ref{hwl}). 
\begin{proof}
The proof of lemma \ref{sad} is lengthy nevertheless straightforward, so we just outline the idea.

Suppose $f$ is conservative. Observe first that a pseudo-pullback is indeed a bilimit (see \cite[6.12]{La}) and recall its explicit expression (see \cite[7.6.3]{B1}): in our case, it is the category whose objects are quintuples $(b,w,h,v,g)$ with $b\in B$, $h\in B^{\mathbbm 2}$, $g\in A^{\mathbbm 2}$, $w:c_B(b)\cong h$, $v:f^{\mathbbm 2}(g)\cong h$, and whose morphisms are triples 
\begin{displaymath}
(x,y,z):(b,w,h,v,g)\Rightarrow(b',w',h',v',g')
\end{displaymath} 
with $x:b\rightarrow b'$, $y:h\Rightarrow h'$ and $z:g\Rightarrow g'$, such that 
\begin{align*}
\label{morpspb}
\begin{split}
 & y\circ w=w'\circ c_B(x) \\
 & y\circ v=v'\circ f^{\mathbbm 2}(z). 
\end{split}
\end{align*} 
Denoting by $B\times^{ps}_{B^{\mathbbm 2}}A^{\mathbbm 2}$ the pseudo-pullback of the diagram in figure \eqref{consdia},
we have an inclusion of $r:A\rightarrow B\times^{ps}_{B^{\mathbbm 2}}A^{\mathbbm 2}$ constructed by means of $f$.
We then define a functor $u:B\times^{ps}_{B^{\mathbbm 2}}A^{\mathbbm 2}\rightarrow A$ as follows: on objects $(b,w,h,v,g)$ in $B\times^{ps}_{B^{\mathbbm 2}}A^{\mathbbm 2}$ we set 
\begin{equation*}
\label{uobj}
u((b,w,h,v,g))=g(0),
\end{equation*}
where $0\in\mathbbm{2}$; on morphisms $(x,y,z):(b,w,h,v,g)\rightarrow(b',w',h',v',g')$ we define 
\begin{equation*}
\label{umor}
u((x,y,z))=z_0,
\end{equation*}
where $z_0$ denotes the natural transformation $z$ computed at $0\in{\mathbbm 2}$.
Clearly $ur=1_A$. That $ru\cong1_{B\times^{ps}_{B^{\mathbbm 2}}A^{\mathbbm 2}}$, and so that the pair $r:A\rightleftarrows B\times^{ps}_{B^{\mathbbm 2}}A^{\mathbbm 2}:u$ is an equivalence and so $A$ a bilimit, 
follows from the hypothesis that $f$ is conservative. We omit however this part.

Concerning the converse, observe first that if \ref{consdia} is a bilimit 
then $(r,u)$ defined above yields an equivalence $A\simeq B\times^{ps}_{B^{\mathbbm 2}}A^{\mathbbm 2}$. Now, if $n:a\rightarrow a'$ is a morphism in $A$ then it defines an object in $A^{\mathbbm 2}$, and, if, in addition, $f(n)$ is also an isomorphism, then it can be extended to an object of $B\times^{ps}_{B^{\mathbbm 2}}A^{\mathbbm 2}$. This finally implies that $n$ is an isomorphism. Again, we omit the details. 
\end{proof}

Lemma \ref{second} and \ref{sad} provide a formulation of axiom 2 in terms of limits.

\begin{lemma}
\label{second2}
The functor $
\D(\CC)\rightarrow\Pi_{D\in\CC}\D(e)$ is conservative if and only if the diagram
\begin{gather}
\begin{aligned}
\xymatrix{
\D(\CC) \ar[rr]
\ar[d]
 & & \D(\CC)^{\mathbbm 2} \ar[d]
 \\
\Pi_{D\in\CC}\D(e) \ar[rr]
 & & (\Pi_{D\in\CC}\D(e))^{\mathbbm 2}
}
\end{aligned}
\label{bipulldia}
\end{gather}
is a bilimit, where arrows are as in diagram \ref{consdia}. 
\end{lemma}

Now, as explained in \ref{ideap}, we have to add to $\Pro$ cones, one for each $\CC\in\mathfrak{Dia}$, which models will then map to the bilimit 
\eqref{bipulldia}, thus forcing them to fulfill axiom 2; the weights defining such cones will have to be cofibrant. 
We proceed as follows. 

For every $\CC\in\mathfrak{Dia}^{op}$, let $\CC'$ denote the category obtained by adjoining an initial object to the discrete category on the objects of $\CC$: in other words, $\CC'$ is the category whose objects are all those of $\CC$ together with a new one $\ast$ acting as initial object, and whose non-trivial morphisms are just the canonical ones with source the initial object $\ast$. 

Given a derivator $\D$, consider the following functors: a diagram 
\begin{equation*}
F_\CC:\CC'\rightarrow\mathfrak{Cat}, 
\end{equation*}
which, on objects, maps $\ast$ to $\D(\CC)$ and the remaining objects to $\D(e)$, and, on morphisms, sends the morphism $\ast\rightarrow C$, for every object $C$ of $\CC$, to the morphism $\D(\CC)\rightarrow\D(e)$, obtained by applying $\D$ to the functor $c_C:e\rightarrow\CC$ in $\mathfrak{Dia}$ constant at $C$ in $\CC$; a weight 
\begin{equation*}
G_\CC:\CC'\rightarrow\mathfrak{Cat}, 
\end{equation*}
which, on objects, maps each $C$ of $\CC$ to $e$ and $\ast$ to ${\mathbbm 2}$, and, on morphisms, takes each $\ast\rightarrow C$ to the canonical morphism ${\mathbbm 2}\rightarrow e$.

We claim that $\{G_\CC,F_\CC\}$ is the bilimit \eqref{bipulldia}. This will imply the following form of axiom 2.

\begin{corollary}
\label{finalform}
The functor $
\D(\CC)\rightarrow\Pi_{D\in\CC}\D(e)$ is conservative if and only if $\D(\CC)\cong\{G_\CC,F_\CC\}$.
\end{corollary}
\begin{proof}
The claim follows from the observation that a natural transformation $G_\CC\Rightarrow\mathfrak{Cat}(\{G_\CC,F_\CC\},F_\CC-)$ consists of: 
\begin{enumerate}
\item a functor $G_\CC(\ast)\rightarrow\mathfrak{Cat}(\{G_\CC,F_\CC\},F_\CC(\ast))$, that is, 
a functor $\{G_\CC,F_\CC\}\rightarrow\D(\CC)^{\mathbbm 2}$; 
\item a functor $G_\CC(C)\rightarrow\mathfrak{Cat}(\{G_\CC,F_\CC\},F_\CC(C))$ for every object $C$ of $\CC$, that is, 
a functor $\{G_\CC,F_\CC\}\rightarrow\D(e)$; 
\item for every arrow $\ast\rightarrow C$ in $\CC'$, with $C\in\CC$, a commutative diagram imposing that each composition 
\begin{equation*}
\{G_\CC,F_\CC\}\rightarrow\D(e)\rightarrow\D(e)^{\mathbbm 2}, 
\end{equation*}
of the functor in (2) with that induced by ${\mathbbm 2}\rightarrow e$, agrees with the composition 
\begin{equation*}
\{G_\CC,F_\CC\}\rightarrow\D(\CC)^{\mathbbm 2}\rightarrow\D(e)^{\mathbbm 2}, 
\end{equation*}
of the functors in (1) with those induced by $c_C:e\rightarrow\CC$;
of such diagram we display below the part defined by $C\in\CC$:
\begin{displaymath}
\xymatrix{
\{G_\CC,F_\CC\} \ar[r] \ar[d] & \D(\CC)^{\mathbbm 2} \ar[d] \\
\D(e) \ar[r] & \D(e)^{\mathbbm 2}.
}
\end{displaymath}
\end{enumerate} 
\end{proof}

In view of corollary \ref{finalform} we have to impose that the bilimit of diagram \eqref{bipulldia}), computed by $\{G_\CC,F_\CC\}$, is $\D(\CC)$. 
To this purpose we consider, for every $\CC\in\mathfrak{Dia^{op}}$, the cone
\begin{equation}
\label{secsec}
(\CC',F'_\CC,G_\CC,\CC,\gamma),
\end{equation}
where $\CC'$ and $G_\CC$ have been defined above; $F'_\CC:\CC'\rightarrow\Gk$ is the functor which, in a way analogous to what $F_\CC$ does, maps $\ast$ to $\CC$ and the remaining objects to $e$, and sends the unique morphism $\ast\rightarrow C$, for every object $C$ of $\CC$, to the morphism in $\Gk$ corresponding to the functor $c_C:e\rightarrow\CC$ in $\mathfrak{Dia}$ constant at $C$ in $\CC$; 
and $\gamma$ is a 2-natural transformation $G_\CC\Rightarrow\Gk(\CC,F'_\CC-)$ determined by two identity arrows $\CC\rightarrow\CC$ with the identity 2-morphism between them, and, for each $C\in\CC$, by the arrow $c_C:\CC\rightarrow e$, where the naturality is expressed by the commutativity of the following  
diagram, of which we display below the part corresponding to $C\in\CC$, 
\begin{displaymath}
\xymatrix{
\CC \ar[d]_{c_C} 
\rrtwocell^{1_C}_{1_C}{1} & & \CC \ar[dll]^{c_C} \\
e & &
}
\end{displaymath}

Finally, we replace this pseudo-cone by the cone defined by the 2-natural transformation $\gamma'$ corresponding to $\gamma$ via the isomorphism \eqref{coheq}  
which comes after taking a cofibrant replacement $\Qc G$ of $G$. We add to $\Pro$ all such cones, for every $C\in\mathfrak{Dia}$.

\subsection{Axiom 3}
\label{third}

If we are constructing a sketch for $\mathfrak{Der}^{r}$, to capture axiom 3
we freely adjoin to $\Gk$ a 1-morphism $u_{(!)}:\CC\rightarrow\DD$ and 2-morphisms $\epsilon_{(u_!)}:u_{(!)}u\Rightarrow1_\CC$, $\eta_{(u_!)}:1_\DD\Rightarrow uu_{(!)}$, for any 1-morphism $u:\DD\rightarrow\CC$ in $\mathfrak{Dia}^{op}$ which has not already a left adjoint. We impose the following diagrams in $\Gk$:
\begin{equation}
\begin{aligned}
\label{relterzo}
 & (u\ast\epsilon_{(u_!)})\circ(\eta_{(u_!)}\ast u)=1_u \\
 & (\epsilon_{(u_!)}\ast u_{(!)})\circ(u_{(!)}\ast\eta_{(u_!)})=1_{u_{(!)}} 
\end{aligned}
\end{equation}
These will ensure the existence of a left adjoint to $\D(u)$, for any model $\D$. 

We remark that if we are instead interested in a sketch for $\mathfrak{Dia}^l$ then we should adjoin, for any  $u:\DD\rightarrow\CC$ in $\mathfrak{Dia}^{op}$ not having a right adjoint, a 1-morphism $u_{(\ast)}:\CC\rightarrow\DD$ and 2-morphisms 
$\epsilon_{(u_\ast)}:uu_{(\ast)}\Rightarrow1_\CC$, $\eta_{(u_\ast)}:1_\DD\Rightarrow u_{(\ast)}u$, together with diagrams
\begin{align*}
 & (u_{(u_\ast)}\ast\epsilon_{(\ast)})\circ(\eta_{(u_\ast)}\ast u_{(\ast)})=1_{u_{(\ast)}} \\
 & (\epsilon_{(u_\ast)}\ast u)\circ(u\ast\eta_{(\ast)})=1_u. 
\end{align*}

If we are constructing a sketch for $\mathfrak{Der}^{rl}$ then all the 1-morphisms, 2-morphisms and relations above should be added.

\subsection{Axiom 4}
\label{fourth}

To capture axiom 4 in the sketch for $\mathfrak{Der}^r$,
for any diagram in $\mathfrak{Dia}^{op}$ of the form
\begin{displaymath}
\xymatrix{
D\backslash\CC \drtwocell<\omit>{\alpha} & \CC \ar[l]_f \\
e \ar[u]^t & \DD \ar[l]^d \ar[u]_u 
}
\end{displaymath}
(see axiom 4 in definition \ref{derivatordef} for the meaning of the symbols),
we add a 2-morphism $\alpha_{bc}^{-1}:du_{(!)}\Rightarrow t_{(!)}f$ and impose the commutativity conditions
\begin{equation}
\begin{aligned}
\label{relquarto}
[(\epsilon_{(t_!)}\ast d\ast u_{(!)})\circ(t_{(!)}\ast\alpha\ast u_{(u_!)})\circ(t_{(!)}\ast f\ast\eta_{(u_!)})]\circ\alpha_{bc}^{-1}=1_{du_{(!)}} \\
\alpha_{bc}^{-1}\circ[(\epsilon_{(t_!)}\ast d\ast u_{(!)})\circ(t_{(!)}\ast\alpha\ast u_{(u_!)})\circ(t_{(!)}\ast f\ast\eta_{(u_!)})]=1_{t_{(!)}f},
\end{aligned}
\end{equation}
provided such a morphism is not already in $\Gk$.

If concerned with $\mathfrak{Der}^l$ or $\mathfrak{Der}^{rl}$, we proceed by adapting what done above to the new situation in the obvious way.

\subsection{Summary}
\label{resumesk}
We summarize the construction of the sketch $\Sk=(\Gk,\Pro)$ for $\mathfrak{Der}^r$.

\subsubsection{Cones}

The set $\Pro$ contains the following cones:
\begin{enumerate}
\item $(\{0,1\},F,\delta_e,\CC_0\amalg\CC_1,(s_{\CC_0},s_{\CC_1}))$, for any objects $\CC_0$ and $\CC_1$ of $\mathfrak{Dia}$ (see \ref{first});
\item $\varnothing$ the empty cone (see \ref{first});
\item $(\CC',F'_\CC,G_\CC,\CC,\gamma')$, for every object $\CC\in\mathfrak{Dia}$ (see \ref{secsec}).
\end{enumerate}

\subsubsection{$\Gk$}

The 2-category $\Gk$ is the free 2-category on $\mathfrak{Dia}^{op}$ with new symbols and with commutativity conditions adjoined. It is made of the following elements:
\begin{enumerate}
\item elements of $\mathfrak{Dia}^{op}$;
\item 1-morphism $u_{(!)}$ and 2-morphisms $\epsilon_{(u_!)}$, $\eta_{(u_!)}$, for every 1-morphism $u\in\mathfrak{Dia}^{op}$ without a left adjoint \ref{third};
\item 2-morphism $\alpha_{bc}^{-1}$, for any 2-morphism $\alpha\in\mathfrak{Dia}^{op}$ as in \ref{fourth};
\item elements obtained as a result of the free construction over the previous elements and the commutativity conditions.
\end{enumerate}

We omit a summary for the  sketches for $\mathfrak{Der}^l$ and $\mathfrak{Der}^{rl}$, which can be obtained from the sketch for $\mathfrak{Der}^r$ by making the proper substitutions or additions, as outlined in \ref{third} and \ref{fourth}.

\begin{remark}
\rm
Observe that conservativity can be expressed not only in terms of the bilimit 
\ref{consdia}, but also by means of the following strict pullback
\begin{gather}
\begin{aligned}
\xymatrix{
A^I \ar[r] \ar[d]_{f^I} \ar[r]^{b_A} & A^{\mathbbm 2} \ar[d]^{f^{\mathbbm 2}} \\
B^I \ar[r]_{b_B} & B^{\mathbbm 2}
}
\end{aligned}
\end{gather}
Since $b_B$ is an isofibration, the pullback above is a homotopy pullback. 

If, in order to capture axiom 2, we construct a sketch with cones for each diagram \ref{consdia}, we will have to introduce a new symbol for $A^I$ and a cone to impose what this symbol should be. However, 
the resulting sketch will be an ordinary 2-sketch, and, since weights are cofibrant, also a homotopy limit 2-sketch. 

If considered as an ordinary 2-sketch, to prove biequivalence between models and derivators, since models preserves products strictly while derivators transform coproducts into products up to equivalence, some rigidification will be necessary. This last problem can be faced also by expressing axiom 1 by means of a suitable strict cone, for every $\CC_0$ and $\CC_1$ in $\mathfrak{Dia}$, and by adjoining an arrow which act as an equivalence between the vertex of such cone and $\CC_0\amalg\CC_1$. 
We could then try to recover 1-morphisms of derivators by restricting to cofibrant models, however, it is not then evident why a cofibrant replacement of a derivator may be identified with some model. Moreover, since the definition of small presentability is up to equivalence, we have preferred a homotopy limit 2-sketch in place of this approach. 

\end{remark}

\subsection{Biequivalence between models and derivators}
\label{verification}

In this subsection we prove that the 2-category $\mathfrak{Der}^r$ is 
biequivalent to the 2-category $\mathfrak{hMod}^{ps}_\Gk$ of models of the homotopy limit 2-sketch $\Gk$. 
If concerned with $\mathfrak{Der}^l$ or $\mathfrak{Der}^{rl}$, the proof is analogous. 

We will exhibit a 2-functor
\begin{equation*} 
\Upsilon:\mathfrak{hMod}^{ps}_\Sk\longrightarrow \mathfrak{Der}^r,
\end{equation*}
and we will outline why $\Upsilon$ is surjective on objects, 
full and faithful on both 1-morphisms and 
2-morphisms, however, omitting those lenghty verifications which looks nevertheless sufficiently clear for the way the sketch $\Sk$ has been constructed. 

\subsubsection{The 2-functor $\Upsilon$}
\label{Upsilon}

Every model $\MM$, via the inclusion $\mathfrak{Dia}^{op}\rightarrow\Gk$, yields a derivator $\Upsilon(\MM)$.

Given any 1-morphism of models
\begin{displaymath} 
\theta=((\theta_\X)_{\X\in\Gk},(\beta^\theta_u)_{u:\X\rightarrow\Y\in\Gk}):\MM_1\rightarrow\MM_2,
\end{displaymath}
consider 
\begin{displaymath} 
\Upsilon(\theta)=((\Upsilon(\theta)_\CC)_{\CC\in\mathfrak{Dia}^{op}},(\beta^{\Upsilon(\theta)}_u)_{u:\CC\rightarrow\DD\in\mathfrak{Dia}^{op}}):\Upsilon(\MM_1)\rightarrow\Upsilon(\MM_2)
\end{displaymath}
where
\begin{displaymath}
\Upsilon(\theta)_\CC=\theta_\CC
\end{displaymath}
for any $\CC\in\mathfrak{Dia}^{op}$, and
\begin{displaymath} 
\beta^{\Upsilon(\theta)}_u=\beta^\theta_u
\end{displaymath}
for any $u\in\mathfrak{Dia}^{op}$. These data do define a morphism of derivators $\Upsilon(\theta)$: 
what is left to prove is that $\Upsilon(\theta)$ is cocontinuous, in other words, that, for any $u\in\mathfrak{Dia}^{op}$, the Beck-Chevalley transform $\beta^{'\theta}_{u_{(!)}}$ of $\beta^\theta_{u}$ is an isomorphism; this can be proved directly by showing that $\beta^{'\theta}_{u_{(!)}}$ coincides with $\beta^{\theta}_{u_{(!)}}$ up to isomorphism, however, we omit the lengthy verification.

Concerning $\Upsilon$ on 2-morphisms,
a modification $\lambda:\theta_1\Rrightarrow\theta_2$ in $\mathfrak{hMod}^{ps}_\Sk$ does define a modification $\Upsilon(\lambda):\Upsilon(\theta_1)\Rrightarrow\Upsilon(\theta_2)$ in $\mathfrak{Der}^r$, by setting for every $\CC\in\mathfrak{Dia}$
\begin{displaymath}
\Upsilon(\lambda)_\CC=\lambda_\CC.
\end{displaymath}

It is now straightforward to check that $\Upsilon$ preserves strictly all compositions and identities, and so it is a 2-functor.

\subsubsection{$\Upsilon$ is surjective on objects}
\label{omega0}

Any derivator $\D$ can be extended along the canonical functor $\mathfrak{Dia}^{op}\rightarrow\Gk$ to a model $\Omega(\D):\Gk\rightarrow\mathfrak{Cat}$ such that $\Upsilon(\Omega(\D))=\D$. 
Indeed, it is enough to assign $\Omega(\D)$ on the symbols adjoined to $\mathfrak{Dia}^{op}$: 
by construction of the sketch $\Sk$, 
this assignment is determined by $\D$ itself; 
for example, 
$\Omega(\D)$ must bring $u_{(!)}$ to a left adjoint $u_!$ to $u^\ast=\D(u)=\Omega(\D)(u)$. 

From this we see that two models determining the same derivators are isomorphic.
\subsubsection{$\Upsilon$ is full and faithful on 1-morphisms}
\label{omega1}

Consider models $\MM_1$ and $\MM_2$ and the corresponding derivators $\Upsilon(\MM_1)$ and $\Upsilon(\MM_2)$. Let $\theta:\Upsilon(\MM_1)\rightarrow\Upsilon(\MM_2)$ be a morphism in $\mathfrak{Der}^r$. We show that we can find a morphism of models $\Omega(\theta):\MM_1\rightarrow\MM_2$ such that $\Upsilon(\Omega(\theta))=\theta$. Let us write
\begin{displaymath}
\theta=((\theta_\CC)_{\CC\in\mathfrak{Dia}^{op}},(\beta^\theta_u)_{u:\CC\rightarrow\DD\in\mathfrak{Dia}^{op}}):\Upsilon(\MM_1)\rightarrow\Upsilon(\MM_2).
\end{displaymath}
We start defining 
\begin{displaymath} 
\Omega(\theta)=((\Omega(\theta)_\X)_{\X\in\Gk},(\beta^{\Omega(\theta)}_u)_{u:\X\rightarrow\Y\in\Gk}):\MM_1\rightarrow\MM_2
\end{displaymath}
by setting $\Omega(\theta)_\X=\theta_X$ for any $X\in\mathfrak{Dia}^{op}$ and $\beta^{\Omega(\theta)}_u=\beta^\theta_u$ for any $u\in\mathfrak{Dia}^{op}$.

We assign now $\Omega(\theta)$ on the symbols adjoined to $\mathfrak{Dia}^{op}$, that is, on $u_{(!)}$, by defining $\beta^{\Omega(\theta)}_{u_{(!)}}$ 
as the Beck-Chevalley transform of $\beta^{\Omega(\theta)}_u$: 
with this definition the naturality of $\beta^{\Omega(\theta)}_{u_{(!)}u}$ and of  $\beta^{\Omega(\theta)}_{uu_{(!)}}$ with respect to $\epsilon_{u_{(!)}}$ and to $\eta_{u_{(!)}}$ respectively, as well as the coherence conditions, are fulfilled; we skip the verification.

The naturality of $\beta^{\Omega(\theta)}_u$ with respect to 2-morphisms of the form $\alpha_{bc}^{-1}$ is also easily verified.

Therefore, $\Upsilon(\Omega(\theta))=\theta$, thus proving that $\Upsilon$ is full on 1-morphisms. 
Since $\beta^{\Omega(\theta)}_{u_{(!)}}$ is completely determined, $\Upsilon$ is also faithful.

\subsubsection{$\Upsilon$ is full and faithful on 2-morphisms}
\label{Omega2}

Consider a modification $\lambda:\Upsilon(\theta_1)\Rrightarrow\Upsilon(\theta_2)$ in $\mathfrak{Der}^r$, where $\theta_1,\theta_2:\MM_1\rightarrow\MM_2$ are 1-morphisms of models. We set 
\begin{displaymath}
\Omega(\lambda)_\CC=\lambda_\CC
\end{displaymath}
for every object $\CC$ in $\mathfrak{Dia}^{op}$. 

The commutativity of diagram \ref{modi} for $u_{(!)}$ follows from commutativity of diagram \ref{modi} for $u$ and the relation between $u$ and $u_{(!)}$ via Beck-Chevalley transforms. 

Since $\Omega(\lambda)$ is completely determined by $\lambda$, then $\Upsilon$ is full and faithful on 2-morphisms.

\section{Homotopy local presentability}
\label{accessibility}

\subsection{Homotopy locally presentable categories}
\label{summaryLR}

We recall some definitions and results from \cite{LR} regarding homotopy local presentability \cite[9.6]{LR} and the characterization  \cite[9.13]{LR}, in the case $\V=\CC at$. 

We recall the definition of homotopy filtered colimit, by means of which we will introduce homotopy presentability \cite[6.4]{LR}.
Let $\lambda$ be a regular cardinal, $\II$ the free 2-category on an ordinary small $\lambda$-filtered category, $F:\II\rightarrow\Ck$ a 2-functor, $\delta_e:\II^{op}\rightarrow\mathfrak{Cat}$ the 2-functor constant at the terminal category, $\Qc\delta_e$ a cofibrant replacement of $\delta_e$: the homotopy $\lambda$-filtered colimit $\rm{hocolim}F$ of $F$ is defined as the weighted homotopy colimit $\Qc\delta_e\star_h F$. Homotopy filtered colimits are computed up to equivalence by ordinary conical filtered colimits 
\cite[5.9]{LR}.

\begin{definition}
Let $\Ck$ be a 2-category. An object $C$ in $\Ck$ is homotopy $\lambda$-presentable if $\Ck(C,-):\Ck\rightarrow\mathfrak{Cat}$ preserves homotopy $\lambda$-filtered colimits.
\end{definition}

The following is the definition of homotopy locally presentable 2-category \cite[9.6]{LR}. Below, a 2-functor $F:\RR\rightarrow\Sk$ is called a local equivalence if $F_{XX'}:\RR(X,X')\rightarrow\Sk(F(X),F(X'))$ is an equivalence of categories for every objects $X$ and $X'$ of $\RR$ (see \cite[7]{LR} or \cite[1.1.4]{R}).

\begin{definition}
\label{defhlp}
Let $\Ck$ be a 2-category admitting weighted homotopy 2-colimits, $i:\Ak\hookrightarrow\Ck$ a small full 2-subcategory of homotopy $\lambda$-presentable objects. We say that $\Ak$ exhibits $\Ck$ as strongly homotopy locally $\lambda$-presentable if every object of $\Ck$ is a homotopy $\lambda$-filtered colimit of objects of $\Ak$. We say that $\Ak$ exhibits $\Ck$ as homotopy locally $\lambda$-presentable if the induced functor
\begin{displaymath}
\xymatrix{
\Ck \ar[r]^(.35){\Ck(i,-)} & [\Ak^{op},\mathfrak{Cat}] \ar[r]^\Qc & [\Ak^{op},\mathfrak{Cat}]
}
\end{displaymath}
is a local equivalence. 

We say that $\Ck$ is strongly homotopy locally $\lambda$-presentable or homotopy locally $\lambda$-presentable if there is some such $\Ak$, and that $\Ck$ is strongly homotopy locally presentable or homotopy locally presentable if it is so for some $\lambda$.
\end{definition}

Notice that strongly homotopy local presentability implies homotopy local presentability (\cite[9.7]{LR}).

A characterization of homotopy locally presentable 2-categories is \cite[9.13]{LR}. 

\begin{theorem}
\label{lrth}
Suppose there exists a combinatorial model 2-category $\Dk$ and a biequivalence $\Ck\rightarrow{\rm Int}\Dk$, then $\Ck$ is strongly homotopy local presentable. Assuming Vop\v{e}nka's principle, the converse holds true, and $\Dk$ can be taken to be a left Bousfield localization of the 2-category $[\Ak^{op},\mathfrak{Cat}]$, where $\Ak$ is as in definition \ref{defhlp}.
\end{theorem}

Note that we will be using only the first part of theorem \ref{lrth} (namely, \cite[9.13]{LR}), which does not depend on Vop\v{e}nka's principle.

\subsection{The 2-category $\mathfrak{hMod}^{ps}_\Sk$ of homotopy models of $\Sk$}
\label{skw}

We now apply what recalled in \ref{summaryLR} to $\mathfrak{hMod}^{ps}_\Sk$. By \cite[9.14(1)]{LR} we know that 
the 2-category of homotopy models of $\Sk$ is homotopy locally presentable, however, as we are interested in $\mathfrak{hMod}^{ps}_\Sk$ where we allow pseudo-natural transformations as 1-morphisms, we show that the same procedure applies also to this case, leading to the same conclusion.

Let ${\rm Int}[\Gk,\mathfrak{Cat}]$ denote the full 2-subcategory spanned by the flexible 2-functors, that is, the cofibrant objects of $[\Gk,\mathfrak{Cat}]$. By means of the cofibrant replacement $\Qc$ (see section \ref{sketches}), we have the following result.

\begin{lemma}
\label{redf}
There is a biequivalence $\Qc:\Pro s(\Gk,\mathfrak{Cat})\longrightarrow{\rm Int}[\Gk,\mathfrak{Cat}]$, provided by the cofibrant replacement functor.
\end{lemma}

We soon deduce the following corollary. 

\begin{corollary}
\label{prohlp}
$\Pro s(\Gk,\mathfrak{Cat})$ is strongly homotopy locally presentable.
\end{corollary}
\begin{proof}
By \cite[9.8]{LR}, ${\rm Int}[\Gk,\mathfrak{Cat}]$ is strongly homotopy locally presentable. The claim now follows now from \ref{redf} and \cite[9.15]{LR}.
\end{proof}

To prove that $\mathfrak{hMod}^{ps}_\Sk$ is homotopy locally presentable, we show that $\mathfrak{hMod}^{ps}_\Sk$ is a homotopy orthogonal subcategory of $\Pro s(\Gk,\mathfrak{Cat})$ (see \cite[4.1]{LR} for the general definition of homotopy orthogonal). 
The proof extends the one given in \cite[6.11]{Ke}. 

\begin{lemma}
\label{dkps}
$\mathfrak{hMod}^{ps}_\Sk$ is a homotopy orthogonal subcategory of $\Pro s(\Gk,\mathfrak{Cat})$.
\end{lemma}
\begin{proof}
Consider a cone $(\E,F,G,\LC,\gamma)\in\Pro$ and the composite, which we denote $i\Y(\gamma)$,
\begin{displaymath}
\xymatrix{
G \ar[r]^(.3)\gamma & \Gk(\LC,F-) \ar[rr]^(.35){i\Y_{\LC,F-}} & & \Pro s(\Gk,\mathfrak{Cat})(i\Y(F-),i\Y(\LC)),
}
\end{displaymath}
where $\Y$ indicates the enriched contravariant Yoneda embedding $\Gk\rightarrow[\Gk,\mathfrak{Cat}]$ and $i$ the inclusion $[\Gk,\mathfrak{Cat}]\hookrightarrow\Pro s(\Gk,\mathfrak{Cat})$.
Since $\Pro s(\Gk,\mathfrak{Cat})$ has weighted homotopy 2-colimits (corollary \ref{prohlp}), $i\Y(\gamma)$ yields a 1-morphism
\begin{equation*}
\rho:G\star_hi\Y(F-)\longrightarrow i\Y(\LC)
\end{equation*}
in $\Pro s(\Gk,\mathfrak{Cat})$. 

We prove that a 2-functor $\MM:\Gk\rightarrow\mathfrak{Cat}$ preserves
the weighted homotopy 2-limits of $\Pro$, that is, it is a homotopy model,
if and only if, for any $\DD\in\mathfrak{Cat}$, the 2-functor $[\DD,\MM-]$ is homotopy orthogonal in $\Pro s(\Gk,\mathfrak{Cat})$ to the collection of 1-morphisms $\rho$ constructed above from cones of $\Pro$, namely,
the functor $\Pro s(\Gk,\mathfrak{Cat})(\rho,[\DD,\MM-])$
\begin{multline}
\label{fff}
\Pro s(\Gk,\mathfrak{Cat})(i\Y(\LC),[\DD,\MM-])\longrightarrow\Pro s(\Gk,\mathfrak{Cat})(G\star_hi\Y(F-),[\DD,\MM-])
\end{multline}
is an equivalence of categories.

Since $\Y(\LC)$ is flexible (\cite[4.6]{BKPS}) and by the enriched Yoneda lemma, we have an equivalence
\begin{equation}
\label{fffs}
\Pro s(\Gk,\mathfrak{Cat})(i\Y(\LC),[\DD,\MM-])\simeq[\DD,\MM(\LC)].
\end{equation}

On the other hand, by definition of weighted homotopy 2-colimit, we obtain an equivalence
\begin{multline*}
\Pro s(\Gk,\mathfrak{Cat})(G\star_hi\Y(F-),[\DD,\MM-])\simeq \\ \simeq[\Gk,\mathfrak{Cat}](G,\Pro s(\Gk,\mathfrak{Cat})(i\Y(F-),[\DD,\MM-])
\end{multline*}
and, using again the flexibility of $\Y(\LC)$ and the enriched Yoneda lemma, an equivalence
\begin{equation}
\label{fffd}
[\Gk,\mathfrak{Cat}](G,\Pro s(\Gk,\mathfrak{Cat})(i\Y(F-),[\DD,\MM-])\simeq[\Gk,\mathfrak{Cat}](G,[\DD,\MM\circ F-].
\end{equation}

By the equivalences \eqref{fffs} and \eqref{fffd}, the functor \eqref{fff} induces an equivalence
\begin{equation*}
[\Gk,\mathfrak{Cat}](G,[\DD,\MM\circ F-]\longrightarrow[\DD,\MM(\LC)],
\end{equation*}
or, equivalently, $\MM(\LC)\simeq\{G,\MM\circ F\}_h$, that is, $\MM$ takes all the cones of $\Pro$ to weighted homotopy limit cones. 
\end{proof}

Writing 
$\Sigma$ for the collection of all morphisms $\rho$ as in lemma \ref{dkps}, $\mathfrak{hMod}^{ps}_\Sk$ can be identified with the homotopy orthogonal subcategory $\Pro s(\Gk,\mathfrak{Cat})_\Sigma$ of $\Pro s(\Gk,\mathfrak{Cat})$. 

\begin{corollary}
\label{nn}
$\mathfrak{hMod}^{ps}_\Sk$ is strongly homotopy locally presentable, and there are biequivalences
\begin{equation*}
\mathfrak{hMod}^{ps}_\Sk\longrightarrow\Pro s(\Gk,\mathfrak{Cat})_\Sigma\longrightarrow{\rm Int}[\Gk,\mathfrak{Cat}]_{\Qc(\Sigma)}
\end{equation*}
\end{corollary}
\begin{proof}
By lemma \ref{redf} and \ref{dkps}, the proof follows from proposition \cite[9.9]{LR}.
\end{proof}

Observe that ${\rm Int}[\Gk,\mathfrak{Cat}]$ and $\Pro s(\Gk,\mathfrak{Cat})$ are strongly homotopy locally finitely presentable, as representable functors are homotopy finitely presentables (see \cite[9.8-7.1(3)]{LR}). We will prove now that $\mathfrak{hMod}^{ps}_\Sk$ is strongly homotopy locally $\lambda$-presentable, where $\lambda$ is a regular cardinal which bounds the size of any category in $\mathfrak{Dia}$. 
First we need a few results summarized in the remark below.

\begin{remark}
\label{lop}
\rm
(1) By \cite[8.5]{LR}, $\mathfrak{hMod}^{ps}_\Sk$ is a homotopy reflective 2-subcategory of $\Pro s(\Gk,\mathfrak{Cat})$.
Let $j$ and $r$ denote the inclusion and reflection 
\begin{equation*}
j:\mathfrak{hMod}^{ps}_\Sk\rightleftarrows\Pro s(\Gk,\mathfrak{Cat}):r.
\end{equation*}
Weighted homotopy 2-colimits in $\mathfrak{hMod}^{ps}_\Sk$ are computed by means of the reflection $r$ from the corresponding weighted homotopy 2-colimit in $\Pro s(\Gk,\mathfrak{Cat})$ (see the proof of \cite[9.9]{LR}): if $F$ is a diagram in $\mathfrak{hMod}^{ps}_\Sk$, then $G\star_hF\simeq r(G\star_hjF)$.

(2) We can now use the biequivalences $\Qc$ and $i$ to compute weighted homotopy 2-colimits in $\Pro s(\Gk,\mathfrak{Cat})$: indeed, by \cite[7.1]{LR} biequivalences preserve and create weighted homotopy colimits, so if $F$ is a diagram in $\Pro s(\Gk,\mathfrak{Cat})$, then $G\star_hF\simeq G\star_hi\Qc F\simeq i(G\star_h\Qc F)$. 

(3) Finally, as explained in the proof of \cite[5.5]{LR}, 
weighted homotopy 2-colimits $G\star_hF$ in ${\rm Int}[\Gk,\mathfrak{Cat}]$ are computed as fibrant replacement of the weighted 2-colimits $G\star F$ in $[\Gk,\mathfrak{Cat}]$, so
by $G\star F$ itself. The advantage is that weighted 2-colimits in $[\Gk,\mathfrak{Cat}]$ are computed pointwise (\cite[3.3]{Ke}).  

(4) It is convenient to replace the weighted homotopy limit 2-sketch $\Sk=(\Gk,\Pro)$ for derivators with a realized one, that is, whose underlying category has the same objects as $\Gk$, whose cones are already homotopy limit cones and whose 2-category of homotopy models is equivalent to that of $\Sk$; the proof of the existence of such homotopy limit 2-sketch is analogous to that of \cite[6.21]{Ke}. We denote this new sketch by $\T$. 
In this way, representable 2-functors, which we will write as $\T(\CC,-)$, are automatically homotopy models of $\T$. 
\end{remark}

Let $\lambda$ be a regular cardinal which bounds the size of any category in $\mathfrak{Dia}$.
\begin{lemma}
\label{hofiltcolim}
Homotopy $\lambda$-filtered colimits in $\mathfrak{hMod}^{ps}_\T$ are computed as in $\Pro s(\Gk,\mathfrak{Cat})$, particularly, they are computed pointwise via $\Qc$.
\end{lemma}
\begin{proof}
Let $\II$ be the free 2-category on an ordinary small $\lambda$-filtered category, and $H:\II\rightarrow\mathfrak{hMod}^{ps}_\T$ a 2-functor. We want to prove that 
the homotopy $\lambda$-filtered colimit ${\rm hocolim}jH$ in $\Pro s(\Gk,\mathfrak{Cat})$ is indeed the homotopy $\lambda$-filtered colimit ${\rm hocolim}H$ in $\mathfrak{hMod}^{ps}_\T$, where $j$ denotes the inclusion of $\mathfrak{hMod}^{ps}_\T$ into $\Pro s(\Gk,\mathfrak{Cat})$. To this purpose, we verify that ${\rm hocolim}jH$ preserves the weighted homotopy limit cones of $\Pro$, thus proving that it belongs to $\mathfrak{hMod}^{ps}_\T$. 


Notice that, as observed in remark \ref{lop}, ${\rm hocolim}jF$ is computed by the pointwise ordinary filtered colimit ${\rm colim}\Qc jH$ in $[\Gk,\mathfrak{Cat}]$. 


Since the weighted homotopy limit cones in $\Pro$ are $\lambda$-small, in the sense that they have $\lambda$-small diagrams and are weighted by $\lambda$-presentable 2-functors, they commute with $\lambda$-filtered homotopy colimits (\cite[6.10]{LR}). Therefore,  
%
\begin{align*}
({\rm colim}\Qc jH)(\{G,F\}_h) & \simeq{\rm colim}(\Qc jH)(\{G,F\}_h)) \\
 & \simeq{\rm colim}(\{G,\Qc jH(F)\}_h) \\
 & \simeq\{G,{\rm colim}(\Qc jH(F))\}_h \\
 & \simeq\{G,({\rm colim}(\Qc jH)(F)\}_h
\end{align*}
\end{proof}

Finally, the next lemma implies that $\lambda$ is a degree of homotopy locally presentability for $\mathfrak{hMod}^{ps}_\T$.

\begin{lemma}
\label{ultimo}
Representable 2-functors on $\T$ 
are homotopy $\lambda$-presentable objects of $\mathfrak{hMod}^{ps}_\T$. The full 2-subcategory of $\mathfrak{hMod}^{ps}_\T$ spanned by $\lambda$-small homotopy 2-colimits of representable models can be taken for the 2-subcategory $\Ak$ in definition \ref{defhlp}.
\end{lemma}
\begin{proof}
By lemma \ref{hofiltcolim} and by the Yoneda lemma for bicategories, representable 2-functors are homotopy $\lambda$-presentable objects of $\mathfrak{hMod}^{ps}_\T$.

Since representable models are cofibrant, we can view them as 2-functors in $[\T,\mathfrak{Cat}]$. From the proof of \cite[9.8]{LR}, we see that 2-functors which are $\lambda$-presentable in $[\T,\mathfrak{Cat}]$ are homotopy $\lambda$-presentable in ${\rm Int}[\T,\mathfrak{Cat}]$. Since $[\T,\mathfrak{Cat}]$ is locally $\lambda$-presentable and representable 2-functors form a set of generators, then every object of $[\T,\mathfrak{Cat}]$ is a $\lambda$-filtered colimit of $\lambda$-small colimits of representables. Therefore, by (3) in \ref{lop}, the full 2-subcategory of ${\rm Int}[\T,\mathfrak{Cat}]$ spanned by $\lambda$-small homotopy 2-colimits of representable models can be taken as $\Ak$ in definition \ref{defhlp} for the homotopy $\lambda$-presentable 2-category ${\rm Int}[\T,\mathfrak{Cat}]$.




By (1) in remark \ref{lop} and lemma \ref{hofiltcolim}, every object of $\mathfrak{hMod}^{ps}_\T$ is a $\lambda$-filtered homotopy colimit of $\lambda$-small homotopy colimits of representables.
\end{proof}

\begin{corollary}
\label{opls1}
$\mathfrak{hMod}^{ps}_\T$ is a homotopy locally $\lambda$-presentable 2-category, where $\lambda$ is a regular cardinal bounding the size of every category in $\mathfrak{Dia}$.
\end{corollary}

\section{Small presentation}
\label{finalres}

In this section we identify representable models for $\T$ with a precise type of derivator, and we prove, via the biequivalence in \ref{Main}, that Renaudin's derivators of small presentation are $\lambda$-presentable objects.

\subsection{Representable models}

Denote by $sSet$ the category of simplicial sets with its classical model structure 
and by $sSet^{\CC^{op}}$ the category of simplicial presheaves endowed with the projective model structure.
Recall that $\Phi$ denotes the pseudo-functor of theorem \ref{cisinski}. The following result is due to Cisinski (see \cite[3.24]{C2}). 

\begin{theorem}
\label{cis}
For every right derivator $\D$ and every small category $\CC$ in $\mathfrak{Dia}$ there is an equivalence of categories 
\begin{displaymath}
\mathfrak{Der}^r(\Phi(sSet^{\CC^{op}}),\D)
\simeq\D(\CC).
\end{displaymath}
\end{theorem}

Before outlining how the equivalence in theorem \ref{cis} is constructed, we rewrite it as follows. 
Setting $\F(\CC)=\Phi(sSet^{\CC^{op}})$, we have
\begin{equation}
\label{ciseqq}
\Psi:\mathfrak{Der}^r(\F(\CC),\D)
\simeq\D(\CC):\Xi.
\end{equation}


%


\begin{remark}
\label{catsset}
\rm Consider the morphism of localizers
\begin{displaymath}
N:(\CC at,W_\infty)\longrightarrow(sSet,W_{sSet}),
\end{displaymath}
where $N:\CC at\rightarrow sSet$ is the nerve and $W_{sSet}$ is the class of weak-equivalences of $sSet$ and $W_\infty=N^{-1}W_{sSet}$. This morphism induces an equivalence between the associated derivators, namely, $\Hc ot_\CC=[-,\CC at^{\CC^{op}}][W^{-1}_\infty]$ and $\F(\CC)$. 
In view of this, we will use the notation $\F(\CC)$ also for $\Hc ot_\CC$. We refer to \cite[1.1]{C2} for more details. 
\end{remark}

We recall now from \cite[3.18]{C2} and \cite{C3}
how equivalence (\ref{ciseqq}) is constructed. 
We describe first the functor
\begin{displaymath}
\Xi:\D(\CC)\longrightarrow\mathfrak{Der}^r(\F(\CC),\D).
\end{displaymath}
For every $h\in\D(\CC)$, we indicate how the pseudo-natural transformation $\Xi(h):\F(\CC)\Rightarrow\D$ is defined, by giving the functors $\Xi(h)_\DD$, for every object $\DD\in\mathfrak{Dia}$, and referring then to \cite{C2} for the rest. 
For $g\in\F(\CC)(\DD)$, let $\nabla g$ and $\int g$ be the Grothendieck fibration and cofibration associated to $g:\DD\times\CC^{op}\rightarrow\CC at$, by fixing $C\in\CC^{op}$ and $D\in\DD$ respectively. Let $\pi(g):\nabla\int g\rightarrow\DD$ and $\varpi(g):\nabla\int g\rightarrow\CC^{op}$ be the projections:  
\begin{displaymath}
\xymatrix{
 & \nabla\int(g) \ar[dl]_{\pi(g)} \ar[dr]^{\varpi(g)} & \\
\DD & & \CC
}
\end{displaymath}
Applying $\D$, we obtain the diagram 
\begin{displaymath}
\xymatrix{
 & \D(\nabla\int(g)) \ar@<-.5ex>[dl]_{\pi(g)_{!}} \ar@<.5ex>[dr]^{\varpi(g)_{!}} & \\
\D(\DD) \ar@<-.5ex>[ur]_{\pi(g)^\ast} & & \D(\CC) \ar@<.5ex>[ul]^{\varpi(g)^\ast}
}
\end{displaymath}
The functor $\Xi(h)_\DD:\F(\CC)(\DD)\rightarrow\D(\DD)$ is defined on objects $g\in\F(\CC)(\DD)$ as
\begin{equation}
\label{cisform}
\pi(g)_{!}\varpi(g)^\ast(h).
\end{equation}
The action of $\Xi(h)_\DD$ on morphisms is as follows: for $\alpha:g\rightarrow g'$ in $\F(\CC)(\DD)$, we set $\beta=\nabla\int\alpha$, yielding in $\mathfrak{Dia}$ the commutative diagram
\begin{displaymath}
\xymatrix{
 & \nabla\int(g)  \ar[dl]_{\pi(g)} \ar[d]^\beta \ar[dr]^{\varpi(g)} & \\
\DD & \nabla\int(g') \ar[l]^(.6){\pi(g')} \ar[r]_(.65){\varpi(g')} & \CC,
}
\end{displaymath}
$\Xi(h)_\DD(\alpha)$ is now defined as the composite
\begin{equation*}
\pi(g)_{!}\varpi(g)^\ast(h)\cong\pi(g')_{!}\beta_!\beta^\ast\varpi(g')^\ast(h)\longrightarrow\pi(g')_{!}\varpi(g')^\ast(h)
\end{equation*}
We refer to \cite[3.19]{C2} to complete the definition of $\Xi(h)$.

We now consider the other functor in (\ref{ciseqq})
\begin{displaymath}
\Psi:\mathfrak{Der}^r(\F(\CC),\D)\longrightarrow\D(\CC). 
\end{displaymath}
As explained in remark \ref{catsset}, we can view the Yoneda embedding $\Y:\CC\rightarrow\CC at^{\CC^{op}}$ as an object of $\F(\CC)(\CC)$. Any 1-morphism of derivators $\theta:\F(\CC)\rightarrow\D$, when computed at $\CC$, yields a functor $\theta_\CC:\F(\CC)(\CC)\rightarrow\D(\CC)$, whose value $\theta_\CC(\Y)$ at $\Y$ defines $\Psi(\theta)$.

We establish now a correspondence between 
derivators $\F(\CC)=\Phi(sSet^{\CC^{op}})$ and representable models $\T(\CC,-)$ of the homotopy limit 
2-sketch $\T$.   

\begin{proposition}
\label{repf}
For every $\CC\in\mathfrak{Dia}$, the derivator $\Upsilon(\T(\CC,-))$ corresponding to the representable model $\T(\CC,-)$ is equivalent in $\mathfrak{Der}^r$ to  
$\F(\CC)$.
\end{proposition}
\begin{proof}
For the way $\Upsilon$ is defined, the derivator $\Upsilon(\T(\CC,-))$ will be reasonably denoted $\T(\CC,-)$. 

On the one hand, equivalence (\ref{ciseqq}) for $\D=\T(\CC,-)$ becomes
\begin{equation}
\label{ciseq2}
\Psi:\mathfrak{Der}^r(\F(\CC),\T(\CC,-))\rightleftarrows\T(\CC,\CC):\Xi.
\end{equation}
Noting that the category $\T(\CC,\CC)$ has $\mathfrak{Dia}^{op}(\CC,\CC)=[\CC,\CC]$ as subcategory, let
\begin{equation*}
\varphi:\F(\CC)\Rightarrow\T(\CC,-)
\end{equation*}
be the 1-morphism of derivators $\Xi(1_{\CC})$.

On the other hand, by the Yoneda lemma for bicategories (see \cite[1.9]{St}) there is an equivalence of categories
\begin{equation}
\label{yoneq}
\Lambda:\mathfrak{hMod}^{ps}_\T(\T(\CC,-),\Omega(\F(\CC)))\rightleftarrows\Omega(\F(\CC))(\CC):\Pi,
\end{equation}
where $\Omega(\F(\CC)$ is any homotopy model such that $\Upsilon\Omega(\F(\CC))\simeq\F(\CC)$ (such models are all equivalent), and, again, we will denote the derivator $\Upsilon\Omega(\F(\CC))$ simply as $\Omega(\F(\CC))$. 
Consider the Yoneda embedding $\Y:\CC\rightarrow\CC at^{\CC^{op}}$ as an object of $\F(\CC)(\CC)$ and, by means of the equivalence above, as an element, which we denote again $\Y$, of $\Omega(\F(\CC))$. Let
\begin{equation*} 
\psi:\T(\CC,-)\Rightarrow\Omega(\F(\CC))
\end{equation*}
be the 1-morphism of models $\Pi(\Y)$: 
for $\DD\in\Gk$ and $g\in\T(\CC,\DD)$
\begin{equation*}
\psi_\DD(g)=\Omega(\F(\CC))(g)(\Y),
\end{equation*}
particularly, when $g:\CC\rightarrow\DD$ is a morphism in $\Gk$ corresponding to some $g:\DD\rightarrow\CC$ in $\mathfrak{Dia}$, then $\psi_\DD(g)=\Y\circ g$. We write $\psi$ also for the morphism of derivators $\Upsilon(\psi)$, and, by the equivalence $\Omega(\F(\CC)\simeq\F(\CC)$, we have $\psi_\DD(g)\cong\Y\circ g$, for $g$ in $\mathfrak{Dia}$.

To prove the lemma we show there are isomorphic modifications $\varphi\circ\psi\Rrightarrow1_{\T(\CC,-)}$ and $\psi\circ\varphi\Rrightarrow1_{\F(\CC)}$.

Formula (\ref{cisform}), for $\DD\in\Gk$ and $g\in\F(\CC)(\DD)$, yields
\begin{equation*}
\varphi_\DD(g)=\T(\CC,\pi(g))_!\T(\CC,\varpi(g))(1_{\CC}),
\end{equation*}
which we visualize in the diagram
\begin{displaymath}
\xymatrix{
 & \T(\CC,\nabla\int(g)) \ar@<-.5ex>[dl]_{\T(\CC,\pi(g)_{(!)})~{}~} \ar@<.5ex>[dr]^{~{}~\T(\CC,\varpi(g)_{(!)})} & \\
\T(\CC,\DD) \ar@<-.5ex>[ur]_{~{}~\T(\CC,\pi(g))} & & \T(\CC,\CC), \ar@<.5ex>[ul]^{\T(\CC,\varpi(g))~{}~}
}
\end{displaymath}
where note that $\T(\CC,\pi(g))_!$ denotes a left adjoint to $\T(\CC,\pi(g))$, and that, viewing $\T(\CC,-)$ as model, $\T(\CC,\pi(g))_!$ equals $\T(\CC,\pi(g)_{(!)})$ up to isomorphism; analogous considerations hold for $\T(\CC,\varpi(g))$ and $\T(\CC,\varpi(g)_{(!)})$.
Notice also that $\T(\CC,\varpi(g))$ acts by composing in $\Gk$ with the projection $\varpi(g)$,
so $\T(\CC,\varpi(g))(1_{\CC})=\varpi(g)$. Similarly $\T(\CC,\pi(g)_{(!)})$ acts by composing in $\Gk$ with $\pi(g)_{(!)}$, therefore 
\begin{equation*}
\varphi_\DD(g)=\pi(g)_{(!)}\varpi(g).
\end{equation*}

As a consequence we find out that
\begin{align*}
\psi\circ\varphi & =\Omega(\F(\CC))(\varphi(-))(\Y) \\
 & =\Omega(\F(\CC))(\pi(-)_{(!)}\varpi(-))(\Y) \\
 & =\Omega(\F(\CC))(\pi(-)_{(!)})\Omega(\F(\CC))(\varpi(-))(\Y).
\end{align*}
So, 
by the equivalence (\ref{ciseqq}), particularly \eqref{cisform}, for $\D=\F(\CC)$, observing the diagram  
\begin{displaymath}
\xymatrix{
 & \F(\CC)(\nabla\int(g)) \ar@<-.5ex>[dl]_{\F(\CC)(\pi(g))_!~{}~{}~} \ar@<.5ex>[dr]^{~{}~{}~\F(\CC)(\varpi(g))_!} & \\
\F(\CC)(\DD) \ar@<-.5ex>[ur]_{~{}~\F(\CC)(\pi(g))} & & \F(\CC)(\CC) \ar@<.5ex>[ul]^{\F(\CC)(\varpi(g))~{}~}
}
\end{displaymath}
with $g\in\F(\CC)(\DD)$,
we see that $\psi\circ\varphi$ is isomorphic to $\Xi(\Y)$; 
on the other hand, the image of the identity $1_{\F(\CC)}\in\mathfrak{Der}^r(\F(\CC),\F(\CC))$ by $\Psi$ is $\Y$; so $\psi\circ\varphi$ and $1_{\F(\CC)}$ are isomorphic in $\mathfrak{Der}^r(\F(\CC),\F(\CC))$, that is, there is an isomorphic modification $\psi\circ\varphi\Rrightarrow1_{\F(\CC)}$. 

As to $\varphi\circ\psi$, observe that
\begin{align*}
\varphi\circ\psi & =\varphi(\Omega(\F(\CC))(-)(\Y)) \\
 & =\pi(\Omega(\F(\CC))(-)(\Y))_{(!)}\varpi(\Omega(\F(\CC))(-)(\Y)).
\end{align*}
The equivalence $\Lambda$ in (\ref{yoneq}) maps $\varphi\circ\psi:\T(\CC,-)\Rightarrow\T(\CC,-)$ to the object $\Lambda(\varphi\circ\psi)$ in $\T(\CC,\CC)$ obtained by computing $\varphi\circ\psi$ at $\CC$ and then evaluating at $1_\CC$:
\begin{equation*}
\pi(\Omega(\F(\CC))(1_\CC)(\Y))_{(!)}\varpi(\Omega(\F(\CC))(1_\CC)(\Y))=\pi(\Y)_{(!)}\varpi(\Y). 
\end{equation*}
This, by lemma 3.22 in \cite{C2}, is isomorphic to the identity $1_\CC$, providing an isomorphic modification $\varphi\circ\psi\Rrightarrow1_{\T(\CC,-)}$.
\end{proof}

As a consequence of lemma \ref{ultimo} and proposition \ref{repf} above we have the following result.

\begin{corollary}
\label{ultimo7}
Any right derivator is a homotopy $\lambda$-filtered colimit in $\mathfrak{Der}^r$ of $\lambda$-small homotopy 2-colimits of derivators of the form $\F(\CC)=\Phi(sSet^{\CC^{op}})$.
\end{corollary}

\subsection{Derivators of small presentation}
\label{finalresren}

Let $\mathfrak{ModQ}^c[\Qc^{-1}]$ be the pseudo-localization at Quillen equivalences $\Qc$ of the 2-category of combinatorial model categories $\mathfrak{ModQ}^c$, as in \cite[2.3]{R}. The following theorem, proved by Renaudin \cite[3.3.2]{R}, builds on Dugger's results on universal homotopy theories \cite{D1} and on presentations of combinatorial model categories \cite{D2}. 

\begin{theorem}
\label{ren1}
The pseudo-functor $\Phi$ induces a local equivalence 
\begin{equation*}
\tilde{\Phi}:\mathfrak{ModQ}^c[\Qc^{-1}]\longrightarrow\mathfrak{Der}_{ad}.
\end{equation*}
\end{theorem}

Renaudin also describes the essential image of $\tilde{\Phi}$: it is formed by derivators of small presentation. 
We recall this result and the relevant definitions from \cite[3.4]{R}.

\begin{definition}
Given a prederivator $\D$, a localization of $\D$ is an adjunction $\theta:\D\rightleftarrows\D':\chi$ such that the counit $\epsilon:\theta\circ\chi\rightarrow 1_{\D'}$ is an isomorphism.
\end{definition}

Derivators are invariant under localization, in the sense that a localization of a derivator is again a derivator (see \cite[4.2]{C2}).

We recall now from \cite[3.4]{R} the concept of presentation in the case of derivators. The motivation comes from Dugger's definitions of homotopically surjective map (\cite[3.1]{D2}) and of presentation of a model category (\cite[1]{D2} or \cite[6.1]{D1}). 
We will observe an analogy between the definition
of a derivator of small presentation (generation) and the definition, by means of the free construction, of a finitely presented (generated) model of an algebraic theory or module over a ring (see for example \cite[3.8.1]{B}). This analogy relies on the use of ``generators" and ``relations".

\begin{definition}
\label{smgen}
A derivator $\D$ has small generation if there is a category $\CC\in\CC at$ and a localization $\F(\CC)\rightleftarrows\D$. 
\end{definition}


\begin{definition}
\label{smpres}
A derivator $\D$ has small presentation if it has a small generation
$\F(\CC)\rightleftarrows\D$ and there is a set $S$ 
of morphisms in $sSet^{\CC^{op}}$, such that the $S$-local equivalences coincide in $\F(\CC)(e)$ with the inverse image of the isomorphisms in $\D(e)$ by the induced functor $\F(\CC)(e)\rightarrow\D(e)$. In this case, we call the pair $(\CC,S)$ a small presentation for $\D$.
\end{definition}

Let $\mathfrak{Der}_{ad}^{fp}$ be the full 2-subcategory of $\mathfrak{Der}_{ad}$ spanned by derivators of small presentation. The next is the main result of \cite{R}.  

\begin{theorem}
There is a biequivalence $\mathfrak{ModQ}^c[\Qc^{-1}]\rightarrow\mathfrak{Der}_{ad}^{fp}$ induced by $\tilde{\Phi}$.
\end{theorem}
\begin{proof}
See \cite[3.4.4]{R}.
\end{proof}

As a consequence we see that a derivator has small presentation if and only if it is equivalent to a derivator of the form ${\Phi(sSet^{\CC^{op}}/S)}$, where $sSet^{\CC^{op}}/S$ denotes the left Bousfield localization of $sSet^{\CC^{op}}$ with respect to $S$.

For algebraic theories, an intrinsic definition of finitely presented model 
consists in requiring that 
the model represents a functor which preserves filtered colimits (see proposition \cite[3.8.14]{B}). 
A similar situation occurs with finitely presented modules over a ring. We would like to see if anything similar holds for derivators of small presentation. To this purpose, we recall from \cite[5.2]{T} the notion of Bousfield localization of derivators, from which we will deduce a reformulation of small presentation.

\begin{definition}
\label{bousfield}
A derivator $\D$ admits a left Bousfield localization by a subset $S$ of $\D(e)$ if there exists a cocontinuous morphism of derivators 
\begin{equation*}
\gamma:\D\longrightarrow L_S\D
\end{equation*}
mapping the elements of $S$ to isomorphisms in $L_S\D(e)$ and such that for any other derivator $\D'$ the morphism $\gamma$ induces an equivalence of categories
\begin{equation*}
\mathfrak{Der}^r(L_S\D,\D')\longrightarrow\mathfrak{Der}^r_S(\D,\D'),
\end{equation*}
where $\mathfrak{Der}^r_S(\D,\D')$ denotes the category of cocontinuous morphisms of derivators which send the elements of $S$ to isomorphisms in $\D'(e)$.
\end{definition}

Small presentation is a special case of Bousfield localization.

\begin{proposition}
\label{bolo}
If $\D$ is a derivator of small presentation $(\CC,S)$, for some category $\CC$ and some set $S$ as in definition \ref{smpres}, then $\D$ is equivalent to the left Bousfield localization $L_S\F(\CC)$. 
\end{proposition}
\begin{proof}
This result, due to Cisinski, is \cite[5.4]{T}.
\end{proof}

We would like now to translate the notions introduced above in terms of models by means of the biequivalence $\Upsilon:\mathfrak{hMod}^{ps}_\T\rightarrow\mathfrak{Der}^r$. 
Note, however, that we can not use this biequivalence to transfer the notion of localization from derivators to models: in general, of the two morphisms forming a localization of derivators only one is a morphism in $\mathfrak{Der}^r$. Nevertheless, we can reformulate finite presentation in terms of models by means of proposition \ref{bolo} as it uses only cocontinuous morphisms.

Observe that, as localizations of categories are coinverters, similarly, derivators of small presentation, regarded as Bousfield localizations, can be written as coinverters. 

\begin{lemma}
\label{tab}
If $\D$ is a derivator of small presentation $(\CC,S)$, then it is equivalent to the coinverter
\begin{displaymath}
\xymatrix{
\D\simeq{\rm coinv}\Big(\T(\tilde{S},-) \rtwocell^s_t{\eta} & \T(\CC,-)\Big),\qquad\qquad
}
\end{displaymath}
computed in $\mathfrak{hMod}_\T^{ps}$, where $\tilde{S}$ is the subcategory of the category of arrows of $\CC$ spanned by $S$ ($s$, $t$ and $\eta$ are defined below in the proof).
\end{lemma}
\begin{proof}
With $\D$ being identified with a 
homotopy model in $\mathfrak{hMod}^{ps}_\T$, by the Yoneda lemma, the diagram
\begin{gather}
\begin{aligned}
\xymatrix{
\T(\tilde{S},-) \rtwocell^s_t{\eta} & \T(\CC,-)
}
\end{aligned}
\label{coinv}
\end{gather}
corresponds in $\T$ to the diagram
\begin{displaymath}
\xymatrix{
\CC \rtwocell^u_v{\alpha} & \tilde{S},
}
\end{displaymath}
where $\T(u,-)=t$, $\T(v,-)=s$ and $\T(\alpha,-)=\eta$. As $\T(\CC,\tilde{S})\simeq{\rm Ho}[\tilde{S}^{op},sSet^{\CC^{op}}]$, the coinverter \eqref{coinv} is completely assigned by choosing $v$ and $u$ to be the obvious source and tail functors, and $\eta$ the canonical natural transformation between them. Since coiverters are PIE-colimits, and so they compute their non-strict counterparts, and by lemma \ref{bolo}, it follows that the universal property of the coinverter \eqref{coinv} is just the universal property of the  left Bousfield localization of derivators $\D\simeq L_S\F(\CC)$.
\end{proof}

\begin{theorem}
\label{main77}
If a model of $\T$ is a Bousfield localization of a representable one, then it is a homotopy $\lambda$-presentable object of $\mathfrak{hMod}^{ps}_\T$.
\end{theorem}
\begin{proof}
Since, by \cite[9.5]{LR}, $\lambda$-small weighted homotopy colimit of homotopy $\lambda$-presentable objects are homotopy $\lambda$-presentable, the theorem follows from lemma \ref{tab}. 
\end{proof}

As a consequence we deduce the following property for small presentation of derivators. 

\begin{theorem}
\label{main7}
A derivator of small presentation is a homotopy $\lambda$-presentable object of $\mathfrak{Der}^r$.
\end{theorem}

\end{document}